\documentclass[]{amsart}

\usepackage{amsfonts}
\usepackage{amscd}
\usepackage{epsfig}
\usepackage{epic,eepic}
\usepackage{graphicx}
\usepackage{psfrag}
\usepackage{amsmath,mathrsfs,amssymb}
\usepackage{amsthm}
\usepackage{setspace}
\usepackage{url}
\usepackage{here}
\usepackage{diagrams}

\newtheorem{theorem}{Theorem}[section]
\newtheorem{lemma}[theorem]{Lemma}
\newtheorem{proposition}[theorem]{Proposition}
\newtheorem{corollary}[theorem]{Corollary}
\newtheorem{question}[theorem]{Question}

\theoremstyle{definition}
\newtheorem{remark}[theorem]{Remark}
\newtheorem{definition}[theorem]{Definition}

\newcommand{\QQ}{\mathbb Q}
\newcommand{\BD}{\mathbf{B}(D)}
\newcommand{\OM}{\overline{\mathcal M}}

\newcommand{\Mgn}{\overline{\mathcal M}_{g, n}}
\newcommand{\PP}{\mathbb P}
\newcommand{\PPr}{\mathbb P^{r}}
\newcommand{\Mrd}{\overline{\mathcal M}_{0,0}(\mathbb P^{r}, d)}
\newcommand{\Mdd}{\overline{\mathcal M}_{0,0}(\mathbb P^{d}, d)}
\newcommand{\Dk}{\Delta_{k, d-k}}
\newcommand{\Mthd}{\overline{\mathcal M}_{0,0}(\mathbb P^{3}, d)}

\newcommand{\Mff}{\overline{\mathcal M}_{0,0}(\mathbb P^{4}, 4)}
\newcommand{\Mrt}{\overline{\mathcal M}_{0,0}(\mathbb P^{r}, 3)}
\newcommand{\Mrf}{\overline{\mathcal M}_{0,0}(\mathbb P^{r}, 4)}
\newcommand{\Mtf}{\overline{\mathcal M}_{0,0}(\mathbb P^{2}, 4)}
\newcommand{\Mthf}{\overline{\mathcal M}_{0,0}(\mathbb P^{3}, 4)}
\newcommand{\Mthth}{\overline{\mathcal M}_{0,0}(\mathbb P^{3}, 3)}
\newcommand{\Mtth}{\overline{\mathcal M}_{0,0}(\mathbb P^{2}, 3)}
\newcommand{\Mtt}{\overline{\mathcal M}_{0,0}(\mathbb P^{2}, 2)}
\newcommand{\Md}{\overline{\mathcal M}_{0,d}}
\newcommand{\Mn}{\overline{\mathcal M}_{0,n}}

\newcommand{\Ma}{\overline{\mathcal M}_{0,\mathcal A}}
\newcommand{\MM}{\overline{\mathcal M}_{0,0}(\mathbb P^{r}, d, d-k)}
\newcommand{\MMone}{\overline{\mathcal M}_{0,0}(\mathbb P^{r}, d, d-1)}
\newcommand{\BV}{\overline{\mathcal B}_{0,0}(\mathbb P^{r}, d)}

\title{Towards Mori's program for the moduli space of stable maps}

\author{Dawei Chen}

\address{University of Illinois at
  Chicago, Department of Mathematics, Statistics and Computer Science, Chicago, IL 60607}

\email{dwchen@math.uic.edu}

\author{Izzet Coskun}

\address{University of Illinois at
  Chicago, Department of Mathematics, Statistics and Computer Science, Chicago, IL 60607}

\email{coskun@math.uic.edu}

\author{with an appendix by Charley Crissman}

\address{University of California at Berkeley, Department of Mathematics, Berkeley, CA 94720}

\email{charleyc@gmail.com}


\thanks{During the preparation of this article the first author was partially supported by MSRI and the second author was
  partially supported by the NSF grant DMS-0737581.}

\begin{document}
\bibliographystyle{plain}

\begin{abstract}
We introduce and compute the class of a number of effective divisors on the moduli space of stable maps $\Mrd$, which, for small $d$, provide a 
good understanding of the extremal rays and the stable base locus decomposition for the effective cone. We also discuss various birational models that arise in Mori's program, including the Hilbert scheme, the Chow variety, the space of $k$-stable maps, the space of branchcurves and the space of semi-stable sheaves.
\end{abstract}

\maketitle

\tableofcontents

\section{Introduction}
This paper studies the Kontsevich moduli space $\Mrd$ from the viewpoint of 
Mori theory. Mori theory for a moduli space consists of the following program: 

\emph{Given a moduli space $\OM$, study its effective
cone, ample cone and base loci of divisors. For an integral divisor $D$
on $\OM$, study the map induced by $D$ and the
model $M(D)=$ Proj $\big(\bigoplus_{m\geq 0} H^{0}(\OM, mD)\big)$. 
Finally, give a geometric interpretation of
the resulting model $M(D)$ and compare it with 
$\OM$.}

For a (log) canonical divisor on an arbitrary variety, this is the standard Mori's program or the (log) minimal model program. Moreover, when the variety is a moduli space, the induced models can often be interpreted as other moduli spaces, which are closely related to the original moduli space.

Hassett initiated the program to study the log canonical models of the Deligne-Mumford moduli space $\Mgn$
of stable genus $g$ curves with $n$-marked points. One can refer to \cite{H1}, \cite{H2}, \cite{HH1}, \cite{HH2}, \cite{HL1}, \cite{HL2}, \cite{Sm}, \cite{mSi}, \cite{FS} and \cite{AS} for related results. 

We can also consider moduli spaces parameterizing curves in an ambient space. The Kontsevich space of stable maps $\Mrd$ is a compactification for the space of degree $d$ rational curves in $\PPr$, cf. \cite{Kon} and \cite{FP}. We work over $\mathbb C$ and always assume that $d,r\geq 2$. 
$\Mrd$ parameterizes isomorphism classes of stable maps $\mu: C\rightarrow \PPr$, where $\mu$ has degree $d$ and $C$ is a connected at-worst-nodal curve of arithmetic genus zero. The map is stable if and only if every component of $C$ contracted by $\mu$ possesses at least three nodes. $\Mrd$ has a nice geometric structure. The intersection theory on $\Mrd$ provides an approach to many enumerative problems in Gromov-Witten theory of rational curves. Here we focus on the divisor theory of $\Mrd$. $\Mrd$ is a smooth Deligne-Mumford stack and its coarse moduli scheme is $\mathbb Q$-factorial with finite quotient singularities. Hence, a Weil divisor on $\Mrd$ is $\QQ$-Cartier.

Let $H$ be the divisor class of maps whose images meet a fixed codimension two linear subspace in $\PPr$. 
Let $\Dk$ be the class of the boundary divisor consisting of maps with reducible domains, where the map has degree $k$ on one component and degree $d-k$ on the other component. For $d, r \geq 2$, Pic($\Mrd)\otimes \QQ$ is generated by $H$ and $\Dk, 1\leq k \leq [d/2]$, cf. \cite{P1}. Let Eff$_{d,r}$ denote the cone of effective divisors on $\Mrd$. For $m > 0$, there is a rational map $\pi: \overline{\mathcal M}_{0,0}(\mathbb P^{r+m}, d)\dashrightarrow \Mrd$ by projecting the image of a map to a linear subspace $\PPr$. $\pi$ is a morphism in codimension one and induces a homomorphism 
$\pi^{*}: Pic(\Mrd)\otimes \QQ \rightarrow Pic(\overline{\mathcal M}_{0,0}(\mathbb P^{r+m}, d))\otimes \QQ$. In \cite{CHS1}, it was proved that 
$\pi^{*}$ induces an injection from Eff$_{d,r}$ to Eff$_{d,r+m}$. Moreover, the injection is a bijection for $r\geq d$. Hence, we use the same notation for an effective divisor on $\Mrd$ and its pull-back on $\overline{\mathcal M}_{0,0}(\PP^{r+m}, d)$. Similarly, we consider Eff$_{d,r}$ as a subcone of Eff$_{d,r+m}$ using this injection. When $d=r$, define $D_{deg}$ as the divisor class of the locus parameterizing maps whose images do not span $\mathbb P^{d}$. It was shown in \cite{CHS1} that a divisor class lies in Eff$_{d,d}$ if and only if it is a non-negative linear combination of $D_{deg}$ and $\Dk$ for $1\leq k \leq [d/2]$. For $d > r$, the structure of Eff$_{d,r}$ is generally unknown. 

Let us introduce more effective divisors. For $r=2$, let $NL$ be the divisor class consisting of maps whose images have a node on a fixed line. Let $TN$ be the divisor class of the locus corresponding to maps whose images have a tacnode. Let $TR$ be the divisor class parameterizing maps whose images have a triple point. All these divisors can be pulled back to effective divisors on $\Mrd$ when $r>2$. Similarly for $r=3$, consider the divisor class $NI$ parameterizing non-isomorphic maps. A general point in $NI$ corresponds to a map whose image is a degree $d$ irreducible rational one-nodal curve in $\mathbb P^{3}$. $NI$ can be pulled back to an effective divisor on $\Mrd$ for $r>3$. Let $T$ denote the divisor class on $\Mrd$ parameterizing maps whose images are tangent to a fixed hyperplane, cf. \cite{P1}, \cite{CHS1} and \cite{CHS2}.
Note that we can define a divisor first in the open locus of maps with smooth domains and then take the closure in $\Mrd$ to obtain a $\QQ$-Cartier divisor.  

One can also define intrinsic divisors as in \cite{DH}. Let $\mathscr{C}$ be the universal curve over the locus $\overline{\mathcal M}^{0}$ of maps that do not have non-trivial automorphisms. The complement of $\overline{\mathcal M}^{0}$ in $\Mrd$ has codimension $\geq 2$ provided $(r,d)\neq (2,2)$. In particular, $\overline{\mathcal M}^{0}$ and $\Mrd$ have the same divisor class group, hence the same rational Picard group. Let $\eta$ be the universal map from $\mathscr{C}$ to $\PPr$, cf. the following diagram: 
$$\begin{CD}
\mathscr C @>{\operatorname{\eta}}>> \PPr \\
@V{\operatorname{\pi}}VV                           \\
\overline{\mathcal M}^{0}
\end{CD}$$
Define a line bundle $\mathscr{L}$ on $\mathscr{C}$ as $\eta^{*}(\mathscr{O}_{\PPr}(1))$. Let $\omega_{\pi}$ denote the relative dualizing sheaf of $\pi$. Use $D$ and $\omega$ to denote the first Chern classes 
of $\mathscr{L}$ and $\omega_{\pi}$, respectively. $D^{2}, D\ldotp \omega$ and $\omega^{2}$ are codimension two in $\mathscr{C}$. Their images under $\pi$ form three divisors in $\overline{\mathcal M}^{0}$: $A = \pi_{*}(D^{2}), \ B = \pi_{*}(D\ldotp\omega), \ C = \pi_{*}(\omega^{2}).$ Let $E_{m} = \mathscr L^{m}$ for $m>0$. $\pi_{*}(E_{m})$ is locally free of rank $1+md$ over $\overline{\mathcal M}^{0}$. We can define a divisor $D_{m} = c_{1}(\pi_{*}(E_{m}))$. Finally, let $K_{\Mrd}$ denote the canonical class of $\Mrd$. 

We first express all the above divisor classes as linear combinations of $H$ and $\Dk$. In order to simplify the expressions, define $\Delta = \sum_{k=1}^{[d/2]}\Dk$ as the total boundary class and define $\Delta_{wt} = \sum_{k=1}^{[d/2]}\frac{k(d-k)}{d}\Dk$ as the weighted boundary class. 

\begin{theorem}
\label{divisor classes}
On $\Mrd$, the divisor classes introduced above have the following expressions:
$$ A = H, $$
$$ B = T - H, $$
$$ C = -\Delta, $$
$$ D_{m} = (\frac{m^{2}}{12}+m) H - mT, $$
$$ K_{\Mrd} = - \frac{(d+1)(r+1)}{2d} H + \frac{r+1}{2}\Delta_{wt} - 2\Delta, $$
$$ D_{deg} =  \frac{d+1}{2d}H - \frac{1}{2}\Delta_{wt}, $$
$$ T = \frac{d-1}{d}H + \Delta_{wt}, $$
$$ NL = \frac{(d-1)(2d-1)}{2d} H - \frac{1}{2}\Delta_{wt}, $$
$$ TN = \frac{3(d-1)(d-3)}{d}H + (d-9) \Delta_{wt} + 4\Delta, $$
$$ TR = \frac{(d-1)(d-2)(d-3)}{2d} H - \frac{d-6}{2}\Delta_{wt} - \Delta, $$
$$ NI = \frac{(d-1)(d-2)}{d}H - \frac{d}{2} \Delta_{wt} + \Delta. $$
\end{theorem}

The next step in the proposed Mori's program is to study the stable base locus decomposition of the effective cone. Note that 
the boundary classes give extremal rays of Eff$_{d,r}$. However, for $d>r$ we do not know the other extremal rays of Eff$_{d,r}$. The structure of Eff$_{d,r}$ can be subtle even for small values of $r$ and $d$. The first unknown case is $d=4$. Use $\langle D_{1}, \ldots, D_{k} \rangle$ to denote the closed cone generated by effective divisors $D_{1}, \ldots, D_{k}$. A face $M$ of a convex cone is a subcone such that if $ u + v \in M$, then $u, v \in M$. In particular, an extremal ray is a one-dimensional face. Use $\BD$ to denote the stable base locus of an effective divisor $D$. $D$ is called moving if $\BD$ does not contain any divisorial components. Our second result reveals all the extremal rays of Eff$_{4, r}$ and provides a detailed study for the cone decomposition. For $d=r=4$, since $D_{deg}, \Delta_{1,3}$ and $\Delta_{2,2}$ generate Eff$_{4,4}$, they are the only possible divisorial components contained in $\BD$ for $D\in$ Eff$_{4,4}$. Introduce two other divisor classes $P = H + \Delta_{1,3} + 4\Delta_{2,2}$ and $Q = 3H + 3\Delta_{1,3} - 2\Delta_{2,2}$ on $\Mff$. See section 2 for a chamber decomposition of Eff$_{4,4}$ by the above divisors. 

\begin{theorem} 
\label{EB}
\ 

(i) Eff$_{4,3}$ is generated by $NI, \Delta_{1,3}$ and $\Delta_{2,2}$; Eff$_{4,2}$ is generated by $TR, TN, \Delta_{1,3}$ and $\Delta_{2,2}$.

(ii) Let $D$ be an effective divisor on $\Mff$. $\BD$ contains $\Delta_{2,2}$ if and only if $D$ lies in the union of the two chambers $\langle P, \Delta_{1,3}, \Delta_{2,2} \rangle\cup \langle D_{deg}, P, \Delta_{2,2}\rangle$; $\BD$ contains $D_{deg}$ if and only if $D$ lies in the 
union of the two chambers $\langle D_{deg}, NI, \Delta_{2,2} \rangle \cup \langle D_{deg}, NI, \Delta_{1,3}\rangle$; $\BD$ contains $\Delta_{1,3}$ if $D$ lies in the union of the two chambers $\langle Q, \Delta_{1,3}, \Delta_{2,2} \rangle\cup \langle D_{deg}, Q, \Delta_{1,3}\rangle$. 
In particular, the cone of moving divisors on $\Mff$ is contained in $\langle NI, Q, P, TR \rangle $. 
\end{theorem}

Let us come back to the final step of our program and consider the model 
$$M(D)= \mbox{Proj}\big(\bigoplus_{m\geq 0} H^{0}(\Mrd, mD)\big).$$
For $d=r=3$, it was proved in \cite{C} that the resulting models are $\Mthth$, the component of the Hilbert scheme parameterizing twisted cubics, the normalization of the Chow component, the closure of the locus of twisted cubics in the space of nets of quadrics and the moduli space of 2-stable maps. Before generalizing the result to $\Mrd$ for arbitrary $d$ and $r$, let us first introduce some special maps in $\Mrd$.  

\begin{definition}
Define a \emph{non-finite map} in $\Mrd$ as a map $[C, \mu]$ such that there is a component $C_{0}$ of $C$ mapping to a point under $\mu$. 
Call a non-finite map $[C, \mu]$ \emph{with moduli} if there exists a connected subcurve $C_{0}$ of $C$ contracted by $\mu$ with the property that $C\backslash C_{0}$ has at least four connected components.  
Define a \emph{degree $k$ tail} of a map $[C, \mu]$ as a component $C_{1}$ of $C$ such that $C_{1}$ meets $\overline{C\backslash C_{1}}$ at one point and the degree of $\mu$ restricted to $C_{1}$ equals $k$. Define a \emph{multi-image map} in $\Mrd$ as a map $[C, \mu]$ such that $\mu$ restricted to a component $C_{2}$ of $C$ is a degree $\geq 2$ multiple cover of its image $\mu (C_{2})$. 
\end{definition}

By the definition, if $[C, \mu]$ is a non-finite map without moduli, every contracted connected subcurve $C_{0}$ has at most three intersection points with
$\overline{C\backslash C_{0}}$. By the stability of the map, $C_{0}$ is a smooth rational curve that meets $\overline{C\backslash C_{0}}$ at three points. 
Three points on $\mathbb P^{1}$ do not have non-isomorphic moduli structures. 

There are a few other modular compactifications for the space of degree $d$ rational curves in $\PP^{r}$. In the Chow
variety parameterizing degree $d$ one-dimensional cycles in $\PPr$, let Chow$_{d,0,r}$ denote the closure of the locus corresponding to integral rational curves. The relation between $\Mrd$ and Chow$_{d,0,r}$ can be interpreted using the divisor class $H$. 

\begin{theorem}
\label{Chow}
There is a forgetful morphism $f: \Mrd \rightarrow Chow_{d,0,r}$ contracting the loci of multi-image maps and non-finite maps with moduli. $f$ factors through the model $M(H)$ and $M(H)$ is the normalization of Chow$_{d,0,r}$. 
\end{theorem}

\begin{corollary}
\label{non-normal}
For $d\geq 3$ and $r\geq 2$, Chow$_{d,0,r}$ is not normal.
\end{corollary}

For $1\leq k\leq [\frac{d-1}{2}]$, let $\MM$ be the moduli space of $k$-stable maps constructed in 
\cite{MM} and \cite{Pa}. It differs from $\Mrd$ in that a degree $e\leq k$ tail of a stable map is replaced by a base point of multiplicity $e$ on the domain curve. There is a morphism $\rho_{k}: \Mrd \rightarrow \MM$ that contracts boundary components $\Delta_{1,d-1}, \ldots, \Dk$. 
This relation between $\Mrd$ and $\MM$ can be revealed using the divisor $T$ for $k=1$ and $\Lambda_{k}(\alpha)$ coming from $S_{d}$-equivariant divisors in the log canonical model program of $\Md$ for $k\geq 2$. 

\begin{theorem}
\label{MM}
For $k = 1$, there is a morphism $t: \Mrd \rightarrow M(T)$ contracting $\Delta_{1,d-1}$ and the locus of non-finite maps with moduli. t  
factors through $\rho_{1}$ and the induced morphism $t_{1}: \MMone\rightarrow M(T)$ only contracts the locus of non-finite maps with moduli. For $k = 2, \ldots, [\frac{d-1}{2}]$ and any $\alpha\in (\frac{2}{k+2}, \frac{2}{k+1}]$, there is a semi-ample divisor $\Lambda_{k}(\alpha)$ on $\Mrd$ such that the induced morphism $\lambda_{k}(\alpha): \Mrd \rightarrow M(\Lambda_{k}(\alpha))$ contracts boundary components $\Delta_{1, d-1}, \ldots, \Delta_{k, d-k}$. In particular, $\lambda_{k}(\alpha)$ factors through $\rho_{k}: \Mrd \rightarrow \MM$.  
\end{theorem}  

In \cite{AK}, the moduli space of branchvarieties parameterizing isomorphism classes of finite morphisms $f$ from reduced (possibly reducible) varieties to $\PPr$ with fixed Hilbert polynomial $\chi(f^{*}\mathcal O_{\PPr}(m))$ for $m \gg 0$ was constructed. Let $\BV$ be the moduli space of branchcurves with Hilbert polynomial $1+md$, i.e. degree $d$ finite maps from reduced connected arithmetic genus zero curves to $\PPr$. The relation between $\Mrd$ and $\BV$, which was observed in \cite[Remark 9.3]{AK} and \cite{AL}, can be interpreted using a positive linear combination of $H$ and $T$.

\begin{theorem}
\label{BV}
For a divisor $D = aH + bT, \ a, b>0$ on $\Mrd$, there is a morphism $\varphi: \Mrd \rightarrow M(D)$ that contracts the locus of non-finite maps with moduli. In particular, there is a morphism $\pi: \Mrd \rightarrow \BV$ that factors through $M(D)$ and $M(D)$ is the normalization of $\BV$. 
\end{theorem}

\begin{corollary}
\label{PSN}
The moduli scheme of branchcurves $\BV$ is projective.
\end{corollary}

Finally, let Simp$^{dm+1}(\PPr)$ be the Simpson moduli space of semi-stable sheaves in $\PPr$ with Hilbert polynomial $dm + 1$, cf. \cite{cSi}. For $r \geq 2$, let $[C,\mu]$ denote a stable map that is not a multi-image map. The direct image $\mu_{*}(\mathcal O_{C})$ is a stable sheaf parameterized by a point of Simp$^{dm+1}(\PPr)$, cf. Proposition~\ref{direct image}. Denote by $S^{dm+1}(\PPr)$ the closure of the locus corresponding to such sheaves in Simp$^{dm+1}(\PPr)$. For $d=3$, we show a wall-crossing phenomenon that flips the space of maps to the space of sheaves. 

\begin{theorem}
\label{Simp}
Let $D = aH + bNL, \ a, b>0$ be a divisor class on $\Mrd$. For $d=3$ and $r\geq 2$, $S^{3m+1}(\PPr)$ is isomorphic to the model $M(D)$, which is the flipping space of $\Mrt$ over the Chow variety with respect to $D$. The locus of multi-image maps is flipped to the locus of sheaves with non-reduced supports. The birational transform of $D$ on $S^{3m+1}(\PPr)$ is ample. 
\end{theorem}

One can run Mori's program for other Kontsevich spaces, for instance, stable maps to the Grassmannians $\mathbb G(k,n)$. See \cite{CC} for a complete discussion when the degree of the map is two or three.

This paper is organized as follows. In section 2, we compute various divisor classes and study in detail 
the effective cone of $\Mrf$. In section 3, we discuss models induced by some of the divisors and their geometric meanings. In section 4, we use alternate spaces to compute the volume of a divisor. In the appendix, a Macaulay 2 code written by Charley Crissman is included to verify certain moving curves used in the proof of Theorem~\ref{EB}. Throughout the paper, we work over $\mathbb C$. We always assume that $d, r \geq 2$. By a divisor we mean a $\QQ$-Cartier divisor. Two divisor classes are equal if they are numerically equivalent. We often consider a divisor class up to a positive scalar, since it does not affect the stable base locus and the induced model. 

{\bf Acknowledgements.} We would like to thank Valery Alexeev, Hua-Liang Chang, Lawrence Ein, David Eisenbud, Maksym Fedorchuk, Joe Harris, Jiayuan Lin, Scott Nollet, Mihnea Popa, Christian Schnell, David Smyth, Jason Starr, Richard Thomas and Fred van der Wyck for help and encouragement during the preparation of this paper. We are especially grateful to Brendan Hassett for many helpful suggetions and to Charley Crissman for designing a related Macaulay program in the appendix. Finally, we want to thank MSRI for providing a wonderful environment for us to finish this paper. 

\section{Divisor classes}
In this section, we prove Theorem~\ref{divisor classes} and discuss the decomposition of the effective cone for $d=4$.

\begin{proof}[Proof of Theorem~\ref{divisor classes}] 
The divisor classes $T$, $K_{\Mrd}$ and $D_{deg}$ have been computed in \cite{P1}, \cite{P2} and \cite{CHS1}, respectively, in a more general setting. The expressions for $NL, TN$ and $TR$ follow from the corresponding formulae in \cite{DH}. 

The relation $A = H$ is directly from the definition of $A$. Note that the self-intersection of the first Chern class of the relative dualizing sheaf on a ruled surface drops by one when we blow up a point to get a node on the corresponding fiber. Therefore, the relation $C = -\Delta$ holds on a general one-parameter family of stable maps. Then it holds on $\Mrd$ as well. For the class of $B$, it suffices to verify that $T = A+B$. Fix a defining hyperplane of $T$ and pull it back by $\eta$ 
to the divisor $D$. The locus $T$ in $\overline{\mathcal M}^{0}$ is the branch locus of the map $\pi$ restricted to $D$. By adjunction, 
$T = \pi_{*} (\omega + D)|_{D} = A + B.$

For $D_{m}$, apply Grothendieck-Riemann-Roch to a one-parameter family of stable maps with smooth general fibers. We get 
$$ ch (\pi_{!} E_{m}) = \pi_{*} (ch(E_{m})\cdot td(\omega_{\pi}^{\vee})). $$
Notice that 
$$td(\omega_{\pi}^{\vee}) = 1 - \frac{\omega}{2} + \frac{\omega^{2}+\delta}{12} + \ldots ,$$
where $\delta$ is the class of nodes. 
We have already seen that $\omega^{2} + \delta = 0$ for such a family. Moreover, $R^{i}\pi_{*} E_{m} = 0$ for $ i > 0 $. 
Therefore, we get 
$$ D_{m} = c_{1}(\pi_{*}E_{m}) = \frac{1}{12}\pi_{*}c_{1}^{2}(E_{m}) - \pi_{*}(\omega\ldotp c_{1}(E_{m})). $$
Since $\pi_{*}c_{1}^{2}(E_{m}) = m^{2}H$ and $\pi_{*}(\omega\ldotp c_{1}(E_{m})) = m B = m (T-H)$, 
we thus obtain $D_{m} = (\frac{m^{2}}{12}+m) H - mT. $

For the divisor $NI$, let us compute its class on $\Mthd$. Assume $NI = a H + \sum_{k=1}^{[d/2]}b_{k}\Dk$. Take a test family $C_{1}$ as a pencil of lines union a fixed degree $d-1$ rational curve at the base point. We have the following 
intersection numbers, 
$$C_{1}\ldotp H = 1, \ C_{1}\ldotp \Delta_{1,d-1} = -1, $$ 
$$C_{1}\ldotp \Delta_{i,d-i} = 0 \ \mbox{for}\ 1<i\leq [d/2], \ C_{1}\ldotp NI = d-2, $$ 
which indicates $ a - b_{1} = d-2. $ 

Take another test family $C_{2}$ as a pencil of conics union a fixed degree $d-2$ rational curve at one of the base points. The intersection numbers are as follows, 
$$C_{2}\ldotp H = 1, \ C_{2}\ldotp \Delta_{1,d-1} = 3, \ C_{2}\ldotp \Delta_{2,d-2} = -1, $$ 
$$C_{2}\ldotp\Delta_{i, d-i} = 0 \ \mbox{for} \ 2<i\leq [d/2], \ C_{2}\ldotp NI = d-3, $$
from which we derive $a+3b_{1} - b_{2} = d-3. $

For $2 < k \leq [d/2]$, let $C_{k}$ be a pencil of degree $k$ rational curves of type $(1, k-1)$ on a fixed quadric surface $Q$ union a fixed 
degree $d-k$ rational curve at one of its base points. Now we have
$$C_{k}\ldotp H = 2, \ C_{k}\ldotp \Delta_{1,d-1} = 2k-3, \ C_{k}\ldotp \Delta_{k-1, d-k+1} = 1, $$ 
$$C_{k}\ldotp \Delta_{k,d-k} = -1, \ C_{k}\ldotp NI = 2(d-k)-1,$$
hence $2a + (2k-3) b_{1} + b_{k-1} - b_{k} = 2d - 2k - 1$. 

Finally, let $C_{0}$ be a pencil of rational curves of type $(1,d-1)$ on $Q$. Since
$$C_{0}\ldotp H = 2, \ C_{0}\ldotp \Delta_{1,d-1} = 2d-2, $$ 
$$C_{0}\ldotp \Delta_{i,d-i} = 0 \ \mbox{for}\ 1<i\leq [d/2], \ C_{0}\ldotp NI = 0,$$
we have $2a + (2d-2) b_{1} = 0$. 

Solving the above equations for $a$ and $b_{i}$, the expression of the divisor class $NI$ follows. 
\end{proof}

When $d=4$, the divisors in Theorem~\ref{divisor classes} specialize to the following: 
$$ T = \frac{3}{4} H + \frac{3}{4}\Delta_{1,3} + \Delta_{2,2}, $$
$$ D_{deg} = \frac{5}{8}H - \frac{3}{8}\Delta_{1,3} - \frac{1}{2}\Delta_{2,2}, $$
$$ NL = \frac{21}{8} H -  \frac{3}{8}\Delta_{1,3} - \frac{1}{2}\Delta_{2,2}, $$
$$ TN = \frac{9}{4}H+\frac{1}{4}\Delta_{1,3}-\Delta_{2,2}, $$
$$ TR = \frac{3}{4}H-\frac{1}{4}\Delta_{1,3}, $$
$$ NI = \frac{3}{2}H-\frac{1}{2}\Delta_{1,3}-\Delta_{2,2}. $$
We need two more divisors: 
$$P = H + \Delta_{1,3} + 4\Delta_{2,2}, $$
$$Q = 3H + 3\Delta_{1,3} - 2\Delta_{2,2}. $$
$P$ spans an extremal ray for the nef cone of $\Mff$. Up to a scalar, $Q$ is defined as the intersection of the two planes spanned by 
$\langle NI, \Delta_{1,3} \rangle$ and $\langle T, \Delta_{2,2} \rangle$. 
Note that $\langle P, H, T \rangle$ generate the nef (semi-ample) cone of $\Mff$, cf. \cite[Remark 5.1]{CHS2}. Moreover, the following groups of divisors are coplanar: 
$$D_{deg}, \ NL, \ H, \ T, \ \Delta_{wt};$$
$$D_{deg}, \ TR, \ P;$$
$$D_{deg}, \ NI, \ TN;$$
$$TR, \ H, \ \Delta_{1,3};$$
$$TR,\ NL,\ TN; $$
$$NI, \ TR, \ \Delta_{2,2}; $$
$$NI, \ Q, \ \Delta_{1,3};$$
$$Q, \ T, \ P,\ \Delta_{2,2}. $$

Up to a scalar, they decompose Eff$_{4,4}$ as follows: 
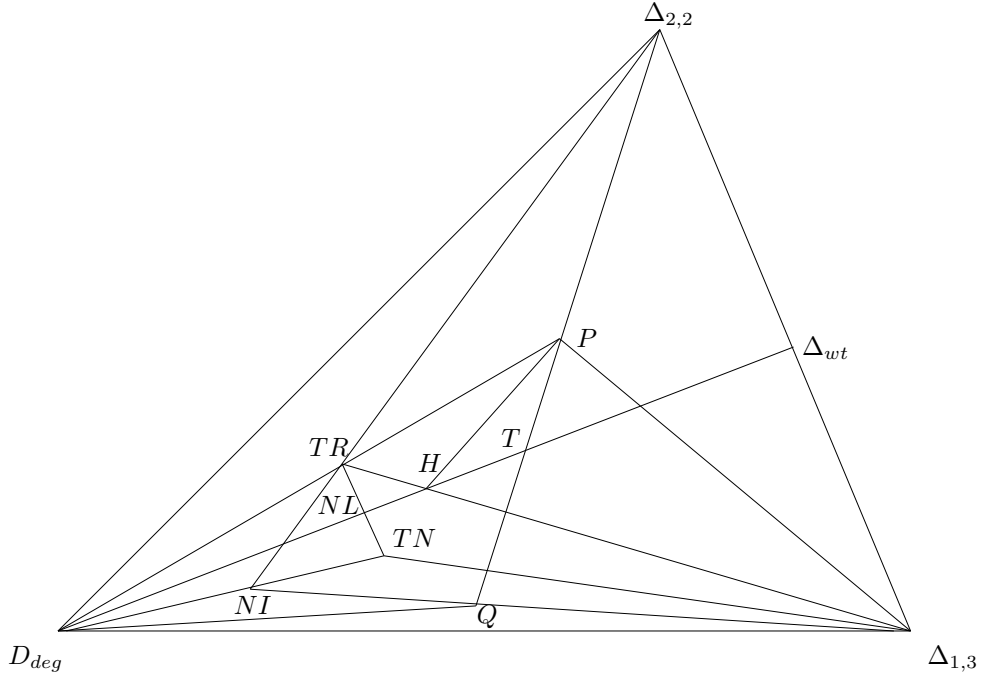
\begin{figure}[H]
\label{Eff4}
    \centering
    \psfrag{Deg}{$D_{deg}$}
    \psfrag{Ni}{$NI$}
    \psfrag{Tn}{$TN$}
    \psfrag{NL}{$NL$}
    \psfrag{Tr}{$TR$}
    \psfrag{H}{$H$}
    \psfrag{T}{$T$}
    \psfrag{P}{$P$}
    \psfrag{D2}{$\Delta_{2,2}$}
    \psfrag{D}{$\Delta_{wt}$}
    \psfrag{D1}{$\Delta_{1,3}$}
    \psfrag{Q}{$Q$}
    \includegraphics[scale=0.7]{Eff4.eps}
\end{figure}

Our strategy to prove Theorem~\ref{EB} is to find moving curves to detect the base locus of a divisor and bound the effective cone. 
An irreducible curve $R$ is moving on a variety $X$ if the deformation of $R$ covers an open dense subset of $X$. 
A moving curve intersects an effective divisor non-negatively. 

We first introduce some curves on $\Mff$. Let $B_{1,3}$ denote a pencil of lines union a fixed twisted cubic at the base point. $B_{1,3}$ is a moving curve on $\Delta_{1,3}$. Let $B_{2,2}$ denote a pencil of plane conics union a fixed conic at a base point. $B_{2,2}$ is a moving curve on $\Delta_{2,2}$.
Let $B_{1}$ be a general pencil of plane rational nodal quartics that pass through three fixed nodes. $B_{1}$ is moving on $\Mtf$. 
Let $B_{2}$ be a general pencil in the linear system $|3F_{1} + F_{2}|$ on a quadric surface in $\mathbb P^{3}$, where $F_{1}$ and $F_{2}$ are the two rulings. $B_{2}$ is a moving curve on $\Mthf$, hence moving on $D_{deg}$. Their intersection numbers with $H$, $\Delta_{1,3}$ and $\Delta_{2,2}$ are as follows: 
$$B_{1,3}\ldotp H = 1, \ B_{1,3}\ldotp \Delta_{1,3} = -1, \ B_{1,3}\ldotp \Delta_{2,2} = 0; $$
$$B_{2,2}\ldotp H = 1, \ B_{2,2}\ldotp \Delta_{1,3} = 3, \ B_{2,2}\ldotp \Delta_{2,2} = -1;$$ 
$$B_{1}\ldotp H = 1, \ B_{1}\ldotp \Delta_{1,3} = 3, \ B_{1}\ldotp \Delta_{2,2} = 3; $$
$$B_{2}\ldotp H = 2, \ B_{2}\ldotp \Delta_{1,3} = 6, \ B_{2}\ldotp \Delta_{2,2} = 0. $$

We need two more complicated moving curves as follows. Take the surface $\mathbb P^{1}\times \mathbb P^{1}$. Let $F_{1}$ and $F_{2}$ denote the two fiber classes. Consider the linear system $|4F_{1}+mF_{2}|$, which has projective dimension $5(m+1) -1$. Its restriction to $F_{2}$ has degree 4. Impose $s = \frac{1}{3}(5(m+1) - 1 - r)$ double points, where $m$ is chosen so that $s$ is an integer. Consider the line bundle $L = 4F_{1} + m F_{2} - 2 \sum_{i=1}^{s}E_{i}$ on $S$, where $S$ is the blow-up of $\mathbb P^{1}\times \mathbb P^{1}$ at $s$ general points and $E_{i}$'s are the exceptional curves. We still use $F_{1}$ and $F_{2}$ to denote their birational transforms on $S$, respectively. If those double points impose independent conditions, the one-parameter family of $F_{2}$ under the image of the linear system $|L|$ provides a one-parameter family $C_{r}$ of rational quartics in $\PPr$. Moreover, we have the intersection numbers:
$$C_{r}\ldotp H = L\ldotp L = 8m - 4s, \ C_{r} \ldotp \Delta_{1,3} = 0, \ C_{r}\ldotp \Delta_{2,2} = s. $$ 
When $r =2$, take $m = 5$, $s = 9$ and we get a curve $C_{2}$. When $r = 3$, take $m = 7$, $s = 12$ and we get another curve $C_{3}$. 
$C_{2}$ and $C_{3}$ satisfy the following intersection numbers: 
$$C_{2}\ldotp H = 4, \ C_{2}\ldotp \Delta_{2,2} = 9, \ C_{2}\ldotp TN = 0;$$
$$C_{3}\ldotp H = 8, \ C_{3}\ldotp \Delta_{2,2} = 12, \ C_{3}\ldotp NI = 0.$$
The fact that the double points impose independent conditions in these two cases can be verified using a checker game described in \cite{Y} and \cite{CHS1}. Note that rational quartic curves in $\mathbb P^{2}$ or $\mathbb P^{3}$ have non-constant moduli up to projective equivalence. We use a Macaulay 2 program by Charley Crissman in the appendix to verify that $C_{2}$ and $C_{3}$ are indeed moving on $\Mtf$ and $\Mthf$, respectively. 

\begin{proof}[Proof of Theorem~\ref{EB} (i)] The moving curve $B_{2}$ on $\Mthf$ has zero intersection with $NI$ and $\Delta_{2,2}$. 
If the sum of two effective divisors $D_{1} + D_{2}$ belongs to the plane spanned by $NI$ and $\Delta_{2,2}$, $B_{2}\ldotp (D_{1} + D_{2}) = 0$. Since $B_{2}$ is moving, $B_{2}$ meets $D_{1}$ and $D_{2}$ non-negatively. Hence, 
$B_{2}\ldotp D_{1} = B_{2}\ldotp D_{2} = 0$. Since the Picard number of $\Mthf$ is three, $D_{1}$ and $D_{2}$ are coplanar with $NI$ and $\Delta_{2,2}$. 
Hence, $NI$ and $\Delta_{2,2}$ span a face of Eff$_{4,3}$. Similarly, the moving curve $C_{3}$ on $\Mthf$ has zero intersection with $NI$ and $\Delta_{1,3}$. Hence, $NI$ and $\Delta_{1,3}$ span a face of Eff$_{4,3}$. Since $NI$ is contained in the intersection of the two faces, $NI$ spans an extremal ray of Eff$_{4,3}$. Therefore, Eff$_{4,3}$ is generated by $NI, \Delta_{1,3}$ and $\Delta_{2,2}$. 

For $\Mtf$, the moving curve $B_{1}$ has zero intersection with $TR$ and $TN$. Hence, $TR$ and $TN$ span a face of Eff$_{4,2}$. 
The moving curve $C_{2}$ on $\Mtf$ has zero intersection with $TN$ and $\Delta_{1,3}$. Hence, $TN$ and $\Delta_{1,3}$ span a face of Eff$_{4,2}$. 
The projection of $B_{2}$ to $\mathbb P^{2}$ is a moving curve on $\Mtf$, which has zero intersection with $\Delta_{2,2}$ and $TR$. Hence, 
$TR$ and $\Delta_{2,2}$ span a face of Eff$_{4,2}$. Therefore, $TR$ and $TN$ are the only extremal rays of Eff$_{4,2}$ in addition to the boundary divisors. Eff$_{4,2}$ is generated by $TR, TN, \Delta_{1,3}$ and $\Delta_{2,2}$.
\end{proof}

\begin{proof}[Proof of Theorem~\ref{EB} (ii)] Write a divisor class in the form $D = aH + b\Delta_{1,3} + c\Delta_{2,2}$. Since $B_{2,2}$ is moving 
on $\Delta_{2,2}$, if $B_{2,2}\ldotp D < 0$, i.e. $a + 3b - c < 0$, $\BD$ has to contain $\Delta_{2,2}$. Note that if $D$ is inside the chamber $\langle D_{deg}, P, \Delta_{2,2}\rangle$, the above inequality holds. For $D$ inside the chamber $\langle \Delta_{2,2}, P, \Delta_{1,3} \rangle$, $B_{1,3}\ldotp D < 0$, so $\BD$ always contains $\Delta_{1,3}$. After removing $\Delta_{1,3}$, a positive linear combination of $P$ and $\Delta_{2,2}$ satisfies the inequality $a + 3b - c < 0$. By contrast, if $D$ lies in the chamber $\langle D_{deg}, P, \Delta_{1,3} \rangle$, since $P$ is eventually base-point-free, 
$\BD$ does not contain $\Delta_{2,2}$.  

Next, $B_{2}$ is a moving curve on $D_{deg}$. Note that $B_{2}\ldotp D < 0$ if and only if $a + 3b < 0$, which implies that $\BD$ contains $D_{deg}$ if $D$ is in the chamber $\langle D_{deg}, NI, \Delta_{2,2}\rangle$. Consider the other moving curve $C_{3}$ on $D_{deg}$. For $D$ in the chamber 
$\langle D_{deg}, NI, \Delta_{1,3}\rangle$, $C_{3}\ldotp D_{deg} < 0$, so $\BD$ contains $D_{deg}$. By contrast, if $D$ lies in the chamber $\langle NI, \Delta_{1,3}, \Delta_{2,2} \rangle$, $\BD$ does not contain $D_{deg}$. 

Finally, $B_{1,3}$ is a moving curve on $\Delta_{1,3}$. $B_{1,3}\ldotp D < 0$ if and only if $a - b < 0$. For a divisor $D$ in the chamber $\langle Q, \Delta_{1,3}, \Delta_{2,2}\rangle$, this inequality holds. So $\BD$ contains $\Delta_{1,3}$. For $D$ in the chamber $\langle D_{deg}, Q, \Delta_{1,3}\rangle$, we have seen that $\BD$ contains $D_{deg}$. After removing $D_{deg}$, a positive linear combination of $Q$ and $\Delta_{1,3}$ 
satisfies the inequality $a - b < 0$. Therefore, $\BD$ contains $\Delta_{1,3}$. 

Note that the complement of the above chambers in Eff$_{4,4}$ is $\langle NI, Q, P, TR \rangle$. Hence, the cone of moving divisors on $\Mff$ is contained in $\langle NI, Q, P, TR \rangle$. 
\end{proof}

\begin{remark}
If one can interpret $Q$ as a geometrically defined divisor without divisorial stable base loci, then the cone of moving divisors on $\Mff$
will be equal to $\langle NI, Q, P, TR \rangle$. 
\end{remark}

\section{Geometric models}

In this section, we discuss the model $M(D)$ of $\Mrd$ induced by a divisor class $D$. We will verify Theorems~\ref{Chow}, \ref{MM}, \ref{BV} and \ref{Simp}. In order to figure out the relation between $M(D)$ and another known moduli space, the following result will be used frequently. 

\begin{lemma} 
\label{proj}
 Let $X$ be a normal projective variety and $L$ be a semi-ample line bundle on $X$. The section ring of $L$ is finitely generated. 
Denote $M(L)=$ Proj \big($\bigoplus_{m\geq 0} H^{0}(X, mL)\big)$ and $f: X\rightarrow M(L)$. Then $f_{*}\mathcal O_{X} \cong \mathcal O_{M(L)}$. 
If there is a morphism $g: X\rightarrow Y$ such that every curve contracted by $f$ is also contracted by $g$, then $g$ factors through 
$f$. If $L'$ is an ample line bundle on $Y$, then $g$ factors through $M(g^{*}L')$ and $M(g^{*}L')$ is the normalization of $Y$. 
\end{lemma}

For a proof of this lemma, one can refer to \cite[2.1.B]{L}.

Before getting into the general discussion, recall what is known for small values of $d$ and $r$. For $r=d=2$, $\Mtt$ is isomorphic to the blow-up of the Hilbert scheme of degree two arithmetic genus zero plane curves along the locus of double lines, i.e. $\Mtt \cong \mbox{Bl}_{V_{2}}\mathbb P^{5}$, where $V_{2}$ is the image of the degree two Veronese embedding $\mathbb P^{2} \hookrightarrow \mathbb P^{5}$. The effective cone 
of $\Mtt$ is generated by $D_{deg}$ and $\Delta=\Delta_{1,1}$. Its nef cone is generated by 
$H$ and $T$. These four divisors provide the desired chamber decomposition. It is easy to check that $M(H) \cong \mathbb P^{5}$ the Hilbert scheme of degree two arithmetic genus zero plane curves. 
For a divisor $D$ in the chamber bounded by $D_{deg}$ and $H$, $\BD=D_{deg}$.  
$\Mtt$ is also isomorphic to the space of complete conics. The map sending a conic to its dual transforms $H$ to $T$ and $D_{deg}$ to $\Delta$. So $M(T)\cong \mathbb P^{5}$ and $\BD=\Delta$ for a divisor $D$ bounded by $T$ and $\Delta$, cf. \cite{CC}. 

Next, let us consider $d=r=3$. The divisors $D_{deg}, NL, H, T$ and $\Delta=\Delta_{1,2}$ provide the stable base locus decomposition of the effective cone. The nef cone is bounded by $H$ and $T$. The model $M(H)$ is the normalization of the corresponding Chow component of twisted cubics. The model $M(T)$ is the space of 2-stable maps $\overline{\mathcal M}_{0,0}(\mathbb P^{3}, 3, 2)$. For a divisor $D$ in the chamber bounded by $H$ and $NL$, $\BD$ is the locus of multi-image maps and $M(D)$ is the corresponding Hilbert component containing twisted cubics. Moreover, 
this Hilbert component is the flip of $\overline{\mathcal M}_{0,0}(\mathbb P^{3}, 3)$ over the Chow variety. $M(NL)$ is the closure of the locus of twisted cubics in the space of nets of quadrics. For a divisor $D$ in the chamber bounded by $NL$ and $D_{deg}$, $\BD= D_{deg}$. If $D$ is contained in the chamber bounded by $T$ and $\Delta$, then 
$\BD = \Delta$, cf. \cite{C}. 

One more example, which can be analyzed using the methods of \cite{C} and \cite{CC}, is $\Mtth$. The effective cone of $\Mtth$ is generated by 
$NL$ and $\Delta$. Again, $H$ and $T$ decompose the cone into three chambers. The nef cone is bounded by $H$ and $T$.
$M(T)$ is the space of 2-stable maps $\overline{\mathcal M}_{0,0}(\mathbb P^{2}, 3, 2)$. For a divisor $D$ in the chamber bounded by $T$ 
and $\Delta$, $\BD = \Delta$. $M(H)$ is the normalization of the discriminant hypersurface in $\mathbb P^{9}$ parameterizing plane nodal cubics and their degenerations. For a divisor $D$ bounded by $H$ and $NL$, $M(D) \subset \mathbb P^{9}\times \mathbb P^{2}$ parameterizes pairs $(C, p)$ where $C$ is a singular plane cubic and $p$ is in the singular locus of $C$. In particular, $M(D)$ is a $\mathbb P^{6}$-bundle over $\mathbb P^{2}$ so it is smooth. Moreover, $NL$ induces a projection to the base $\mathbb P^{2}$ such that the pull-back of $\mathcal O_{\mathbb P^{2}}(1)$ is equivalent to $NL$. We thus obtain that $M(NL) \cong \mathbb P^{2}.$ 

After these examples, let us prove Theorem~\ref{Chow}.

\begin{proof}[Proof of Theorem~\ref{Chow}]
Take a multi-image map and vary a branch point on its image to get a one-dimensional family $R_{1}$ of multi-image maps. All maps parameterized by $R_{1}$ have the same image, so $R_{1}\ldotp H = 0$. Similarly, take a general non-finite map with moduli. Vary other domain components along the contracted subcurve to obtain a one-dimensional family $R_{2}$ of non-finite maps. All maps parameterized by $R_{2}$ have the same image, so $R_{2}\ldotp H = 0$. Conversely, if an irreducible curve $R$ has zero intersection with $H$, then the images of maps parameterized by $R$ cannot vary. Otherwise their images span a surface in $\mathbb P^{r}$ that meets a general codimension two linear subspace at finitely many points, which implies that $R\ldotp H > 0$, a contradiction. Hence, maps parameterized by $R$ are either multi-image maps or non-finite maps. If the maps parameterized by $R$ are non-finite, in order for these maps to have the same image, there exist contracted subcurves of the maps that have different moduli in this family. Since three points on a rational curve have unique moduli, this subcurve of each map has to meet the rest of the domain curve at four points or more. Hence, $R$ is a family of non-finite maps with moduli. 

To see the connection between $M(H)$ and Chow$_{d,0,r}$, note that there is a natural morphism $f: \Mrd \rightarrow$ Chow$_{d,0,r}$ that forgets a map and only remembers its image as a cycle class with multiplicity. Since $\Mrd$ is normal, $f$ is well-defined, cf. \cite[Chapter I]{Kol}. The defining ample divisor on the Chow variety pulls back to $H$ on $\Mrd$ up to a scalar, cf. \cite[1.a]{H}. Then the desired statement follows from Lemma~\ref{proj}.  
\end{proof}

This implies that Chow$_{d,0,r}$ is not normal for $d\geq 3$ and $r\geq 2$. 

\begin{proof}[Proof of Corollary~\ref{non-normal}]
Consider the union of a line $L$ with a general degree $d-1$ rational nodal curve $C$ on a plane. As a Chow point, it is unique. However, $L$ meets $C$ 
at $d-1\geq 2$ points. Therefore, its pre-image under $f$ consists of $d-1$ distinct maps in $\Mrd$, which implies that Chow$_{d,0,r}$ is not normal by Zariski's Main Theorem, cf. \cite[Corollary 11.4]{Ha}.  
\end{proof}

Before proving Theorem~\ref{MM}, let us consider the following example to get a feel of the space of $k$-stable maps $\MM$. Suppose there is a 
one-parameter family of degree $d$ maps from $\mathbb P^{1}$ to $\mathbb P^{r}$ given by 
$[F_{0}(t), \ldots, F_{r}(t)]$ over a base $B =$ Spec $\mathbb C[t]$, where $F_{i}(t)$'s are degree $d$ homogeneous polynomials in two variables $X, Y$. Assume that $F_{0}(0), \ldots, F_{r}(0)$ have a common zero point $p$ of multiplicity $e\leq k$ on the central fiber. Then the total map $\mathbb P^{1} \times B\dashrightarrow \PPr$ is only a rational map, since it is not defined at $p$. If this map can be resolved by blowing up $p$, then the central fiber would become a union of two components mapping with degree $d-e$ and $e$, respectively, which can be parameterized by a point in the boundary $\Delta_{e,d-e}$ of $\Mrd$. On the other hand, we can also keep $p$ as a base point of multiplicity $e$ on the central fiber $C_{0}$ and not blow it up. By removing the common factors, the resulting map has degree $d-e$ on $C_{0}$. Correspondingly, the map restricted to $C_{0}$ along with the base point $p$ of multiplicity $e$ is parameterized by a point in the boundary of $\MM$. There is a natural morphism $\rho_{k}:\Mrd \rightarrow \MM$ that replaces a tail of degree $e \leq k$ by a base point of multiplicity $e$, cf. \cite[Proposition 1.3]{MM}. For $k=1$, the morphism $\rho_{1}$ can be further analyzed using the divisor $T$.

\begin{proposition}
\label{tangent}
The morphism $t: \Mrd \rightarrow M(T)$ contracts $\Delta_{1,d-1}$ and the locus of non-finite maps with moduli. It  
factors through $\rho_{1}$ and the induced morphism $t_{1}: \MMone\rightarrow M(T)$ only contracts the locus of non-finite maps with moduli. 
\end{proposition}

\begin{proof}
For a curve class $R_{1}$ contracted by $\rho_{1}$, by the definition of $\MMone$, all maps parameterized by $R_{1}$ only differ by degree one tails passing through a common attachment point. Choose a defining hyperplane of $T$ which is away from those finitely many attachment points and also not tangent to the remaining part of the image. Then $R_{1}$ does not meet $T$. Hence, $R_{1}$ is contracted by $t$. By Lemma~\ref{proj}, $t$ factors through $\rho_{1}$. 

Conversely, for a curve $R_{2}$ contracted by $t$, the image of a degree $\geq 2$ component of a map parameterized by $R_{2}$ cannot vary. Otherwise we can choose a defining hyperplane of $T$ that is only tangent to finitely many images of the maps parameterized by $R_{2}$. Then $R_{2}\ldotp T$ is positive and $R_{2}$ cannot be contracted under $t$, contradiction. Hence, only the image of degree one tails can vary. In addition, we can also take a non-finite map with moduli and move other domain components along the contracted subcurve. Such a family of non-finite maps is contracted by $t$. But it is not contracted by $\rho_{1}$ since the domain curves in this family have different moduli. 
\end{proof}

For $k\geq 2$, let us explain how to construct the divisors $\Lambda_{k}(\alpha)$ in Theorem~\ref{MM}. 
In \cite[Theorem 1.1]{CHS2}, there is an injection 
$$ v: Pic(\Md)^{S_{d}}\otimes \QQ \rightarrow Pic(\Mrd)\otimes \QQ, $$
where $S_{d}$ denotes the symmetric group on $d$ letters.  
This injection is induced by a rational map $f: \Mrd \dashrightarrow \Md^{S_{d}}$. Take a hyperplane $\Lambda$ to intersect the image of a stable map. If the pre-image of the intersection consists of $d$ distinct points, the stable limit of the domain curve marked at these $d$ points corresponds to a point in $\Md^{S_{d}}$. In particular, $v$ maps $S_{d}$-equivariant nef (semi-ample) divisors on $\Md$ 
to nef (semi-ample) divisors on $\Mrd$. Let $D_{i}$ be the $S_{d}$-equivariant boundary divisor of $\Md$ whose general points parameterize a nodal union of two rational curves with $i$ and $d-i$ marked points, respectively, $2 \leq i \leq [d/2]$. $v$ acts on these boundary classes as follows: 
$$v(D_{i}) = \Delta_{i,d-i} \ \mbox{for} \ i > 2 \ \mbox{and} \ v(D_{2}) = \frac{1}{2}T + \Delta_{2,d-2},$$
cf. \cite[Figure 4]{CHS2}.  

Let $\Ma$ be the moduli space of weighted stable $d$-pointed curves with symmetric weight $\mathcal A = \{a, a, \ldots, a\}$, where 
$\frac{1}{k+1} < a \leq \frac{1}{k}$. In \cite{H1}, it was shown that there is a birational morphism $\Md \rightarrow \Ma$ that contracts boundary components $D_{i}$ with $i \neq 2$ and $ia < 1$. Let $D = \sum D_{i}$ be the total boundary class of $\Md$. Define the model
$$ \Md (\alpha) = \mbox{Proj} \Big(\bigoplus\limits_{m\geq 0} H^{0}(\Md, m(K_{\Md}+\alpha D))\Big). $$
In \cite{mSi}, \cite{FS} and \cite{AS}, the following result was established:  
if $\frac{2}{k+2}<\alpha\leq \frac{2}{k+1}$ for 
some $k = 1, \ldots, [\frac{d-1}{2}]$, then $\Md (\alpha) \cong \Ma$ with $\mathcal A = \{1/k, \ldots, 1/k  \}; $
if $\frac{2}{d-1} < \alpha \leq \frac{2}{[d/2]+1}$, then $\Md (\alpha) \cong (\mathbb P^{1})^{n} // SL_{2}$. 
Consequently, the divisor $K_{\Ma} + \alpha E$ is ample on $\Ma$ in the above range, where $E$ is the total boundary of $\Ma$. Then the pull-back of $K_{\Ma} + \alpha E$ is a semi-ample $S_{d}$-equivariant divisor $A_{\alpha}$ on $\Md$. $A_{\alpha}$ has the following expression: 
$$ A_{\alpha} = \sum_{i = 2}^{k} \Big( \frac{i(i-1)}{2} \alpha - \frac{i(i-1)}{d-1}  \Big) D_{i}  + \sum^{[d/2]}_{i>k} \Big( \frac{i(d-i)}{d-1} - 2 + \alpha \Big) D_{i} .$$ 
Using the map $v$ above, we thus get a semi-ample divisor $\Lambda_{k}(\alpha) = v(A_{\alpha})$ on $\Mrd$ with the following expression: 
$$ (\frac{\alpha}{2} - \frac{1}{d-1}) T + \sum_{i=2}^{k}\Big( \frac{i(i-1)}{2} \alpha - \frac{i(i-1)}{d-1}  \Big) \Delta_{i,d-i} + \sum^{[d/2]}_{i>k} \Big( \frac{i(d-i)}{d-1} - 2 + \alpha \Big) \Delta_{i,d-i} .$$

\begin{proposition}
\label{k2}
For $k = 1, \ldots, [\frac{d-1}{2}]$ and $\frac{2}{k+2}<\alpha\leq \frac{2}{k+1}$,
the morphism $\lambda_{k}(\alpha): \Mrd \rightarrow M(\Lambda_{k}(\alpha))$ induced by the divisor $\Lambda_{k}(\alpha)$
contracts the boundary components $\Delta_{1, d-1}, \ldots, \Delta_{k, d-k}$. 
In particular, $\lambda_{k}(\alpha)$ factors through $\rho_{k}: \Mrd \rightarrow \MM$. 
\end{proposition}

\begin{proof}
There is another way to obtain $\Lambda_{k}(\alpha)$ by the following diagram: 
\begin{diagram}
\Mrd              & \rTo^{\rho_{k}}  &\MM \\
\dDashto^{f}     &           &\dDashto^{g} \\
\Md^{S^{d}} & \rTo^{\pi} & \Ma^{S_{d}}
\end{diagram}

The horizontal morphisms $\rho_{k}$ and $\pi$ contract $\Delta_{1,d-1}, \ldots, \Dk$ and $D_{3},\ldots, D_{k}$, respectively. The vertical maps are only rational. $f$ was constructed in \cite{CHS2} by taking a general hyperplane to slice the image of a map and marking the pre-image of the intersection to get a $d$-marked rational curve. $g$ can be constructed in the same vein. We only need to assign weight $\frac{1}{k}$ to the pre-image of each intersection point. If there is a base point of multiplicity $i\leq k$ on the domain curve, we assign weight $\frac{i}{k}$ to it. For any stable map $[C,\mu]$, there exists a hyperplane such that the map $f$ 
is a local morphism around $[C, \mu]$, cf. \cite{CHS2}. From this fact, it follows that $f$ pulls back $S_{d}$-equivariant semi-ample divisors on $\Md$ 
to semi-ample divisors on $\Mrd$. The same conclusion holds for $g$. Moreover, $\Ma$ is smooth, so $\Ma^{S_{d}}$ has finite quotient singularities. Therefore, we can alternatively obtain $\Lambda_{k}(\alpha)$ by pulling back the ample divisor $K_{\Ma} + \alpha E$ via $\MM$. Note that $g^{*}(K_{\Ma} + \alpha E)$ is semi-ample on $\MM$. Hence, by Lemma~\ref{proj}, the morphism
$\Mrd\rightarrow M(\Lambda_{k}(\alpha))$ factors through $\MM$. 
\end{proof}

\begin{remark}
For $k = 1$, $\Lambda_{1}(\alpha)$ is proportional to $T$, which recovers the result in Proposition~\ref{tangent}. 
\end{remark}

Proposition~\ref{tangent} and \ref{k2} together complete the proof of Theorem~\ref{MM}. 

\begin{corollary} 
\label{push}
\ 

(i) The rational Picard group of $\MM$ is generated by $\rho_{k *}H$ and $\rho_{k *}\Delta_{j, d-j},  k < j \leq [d/2]$.

(ii)  The following push-down pull-back formulae hold: 
$$\rho_{k}^{*}\rho_{k *}\Delta_{j, d-j} = \Delta_{j, d-j}\  \mbox{for} \ j > k \ \mbox{and}\  \rho_{k}^{*}\rho_{k *}H = H + \sum_{i=1}^{k} i^{2} \Delta_{i,k-i}. $$ 

(iii) The canonical class of $\MM$ has the expression: 
$$ K_{\MM} = -\frac{(r+1)(d+1)}{2}\rho_{k *}H + \sum\limits_{j=k+1}^{[d/2]}\Big(\frac{(r+1)j(d-j)}{2d} - 2\Big) \rho_{k *}\Delta_{j,d-j}. $$
\end{corollary}

\begin{proof}
(i) can be seen from the fact that $\rho_{k}$ contracts exactly boundary components $\Delta_{i, d-i}$ for $i \leq k$. Note that the images of $H$ and $\Delta_{j, d-j}$ for $j > k$ under $\rho_{k}$ are still divisors on $\MM$. But the image of $\Delta_{i,d-i}$ for $i \leq k$ has higher codimension in $\MM$. An alternative verification can be done by a direct analysis as in \cite{P1}. 

For (ii), take a pencil of lines on a plane and embed this plane by a degree $i$ Veronese map to $\mathbb P^{N}$. Attach a degree $d-i$ rational curve at the image of the base point of this pencil. Project the whole family to $\mathbb P^{r}$ and we obtain a one-dimensional family $C_{i}$ of maps contained in $\Delta_{i,d-i}$.
For $i\leq k$, the curve class $C_{i}$ is contracted by $\rho_{k}$. 
Hence, $C_{i}\ldotp \rho_{k}^{*}\rho_{k*}(D) = 0$ for any divisor $D$ on $\Mrd$. Note that $C_{i} \ldotp H = i^{2}, C_{i}\ldotp \Delta_{j, d-j} = 0$ for $j\neq i$
and $C_{i}\ldotp \Delta_{i,d-i} = -1$. Then the formulae in (ii) follow immediately. 

Finally, for the canonical class of $\MM$, one can follow Pandharipande's calculation for the canonical class of $\Mrd$, cf. \cite[2.3]{P2}. 
Take a rational map $\mu: S = \mathbb P^{1} \times C \dashrightarrow \PPr$ with simple base points of degree $e$ such that $k < e \leq [d/2]$, cf. \cite[2.1]{P2}. 
In this setting, everything works identically except that boundary divisors $\Delta_{i,d-i}$ for $i\leq k$ do not appear in the expression of $K_{\MM}$. 
\end{proof}

Next we will prove Theorem~\ref{BV}. Let $D = aH + bT, \ a, b > 0$ be an effective divisor on $\Mrd$. 
$D$ is semi-ample since $H$ and $T$ are semi-ample. 

\begin{lemma}
\label{non-finite}
The morphism $\Mrd\rightarrow M(D)$ only contracts the locus of non-finite maps with moduli. 
\end{lemma}

\begin{proof} 
If an irreducible curve $R$ on $\Mrd$ has zero intersection with $D$, it must satisfy $R\ldotp H = 0$ and $R\ldotp T=0$, which implies that the maps parameterized by $R$ have the same image and the branch points of the maps cannot vary. The only possible way to obtain a positive dimensional family of such maps is to have a domain component $C_{0}$ mapping to a point, so that other components can move along $C_{0}$ to vary the moduli of the domain curve. Hence, those maps in this family are non-finite maps with moduli.
\end{proof}

Recall that the moduli space of branchvarieties constructed in \cite{AK} parameterizes isomorphism classes of finite maps $f$ from a connected (reduced but possibly reducible) variety to $\PPr$ such that the Hilbert polynomial of $f^{*}\mathcal O_{\PPr}(m)$ is fixed for $m \gg 0$. We are interested in the finite maps with the fixed Hilbert polynomial $1 + md$. By Riemann-Roch, the domain curve has arithmetic genus zero. This forces the curve to possess a special class of singularities. 

\begin{definition}
A curve singularity $p$ is a \emph{rational $k$-fold point} if it is locally cut out by $k$ smooth branches whose tangent line directions are linearly independent at $p$. 
\end{definition}

For instance, the $k$ coordinate axes of $\mathbb A^{k}$ meeting at the origin form a rational $k$-fold point. Locally, a rational $k$-fold point is isomorphic to this form. 

\begin{lemma}
\label{3.9}
Let $C$ be a connected reduced curve of arithmetic genus zero. Then every irreducible component of $C$ is a smooth rational curve. Moreover, $C$ can only have rational $k$-fold points as its singularities. 
\end{lemma}

\begin{proof}
Let $\pi: \widetilde{C}\rightarrow C$ be the normalization of $C$. We have the following short exact sequence: 
$$ 0 \rightarrow \mathcal O_{C} \rightarrow \pi_{*}\mathcal O_{\widetilde C} \rightarrow \mathcal F \rightarrow 0, $$
where $\mathcal F$ is a sheaf supported on the singularities of $C$. Note that $h^{1}(\mathcal O_{C}) = p_{a}(C) = 0$ and $h^{1}(\mathcal F) = 0$. Thus $h^{1}(\mathcal O_{\widetilde C}) = 0$ and $p_{a}(\widetilde C) = 0$. It implies that every component of $C$ is rational. An irreducible component of $C$ still has arithmetic genus zero. Since it is rational, it must be isomorphic to $\mathbb P^{1}$. 

To see how those components meet at a singularity, let us take a singular point $p$ of $C$. Suppose $k$ branches of smooth rational curves $C_{1}, \ldots, C_{k}$ pass through $p$. Denote their union by $C'$. The arithmetic genus $p_{a}(C')$ is zero. Let $L_{i}$ be a line bundle on $C'$ such that $L_{i}|_{C_{j}}$ is trivial for $j\neq i$ and $L_{i}|_{C_{i}} \cong \mathcal O_{\mathbb P^{1}}(1)$. 
Let $L = L_{1}\otimes \cdots \otimes L_{k}$. By Riemann-Roch, $\chi (L) = 1 - p_{a}(C') + k = 1 + k, $
so $h^{0}(L) \geq k+1$. On the other hand, $h^{0}(L) \leq k+1$ since $h^{0}(L_{i}) = 2$ and those sections over each $C_{i}$ can only be glued to form a global section of $L$ if they take the same value on the fiber 
over $p$, which imposes at least $k-1$ conditions. Therefore, $h^{0}(L) = k+1$ and $h^{1}(L) = 0$. $C'$ maps into $\mathbb P^{k}$ by the linear system $|L|$ and each component $C_{i}$ maps isomorphically to a line. These $k$ lines 
span $\mathbb P^{k}$. Hence, $C_{1}, \ldots, C_{k}$ have linearly independent tangent directions at $p$. 
\end{proof}

By Lemma~\ref{3.9}, we make the following definition. 

\begin{definition}
A connected reduced curve is a \emph{multi-nodal rational curve} if it has arithmetic genus zero.  
\end{definition}

Use $\BV$ to denote the moduli space of isomorphism classes of finite maps from multi-nodal rational curves to $\PPr$ with $1+md$ as the Hilbert polynomial. $\BV$ is a Deligne-Mumford moduli stack over a characteristic zero ground field, cf. \cite[Remark 3.3]{AK}. 
The following result shows that on a set-theoretic level, $\Mrd$ admits a map to $\BV$ that contracts the locus of non-finite maps with moduli. 

\begin{lemma}
\label{local}
Let $[C, \mu] \in \Mrd$ be a stable map and $L = \mu^{*}(\mathcal O_{\mathbb P_{r}}(1))$ be a semi-ample line bundle on $C$. Denote
$C' = \mbox{Proj}\big(\bigoplus_{m\geq 0} H^{0}(C, mL)\big)$ and $f: C\rightarrow C'$.  
If $\mu$ contracts a component $C_{0}$ of $C$, $C_{0}$ is also contracted by $f$ and the resulting singularity in $C'$ is a rational $k$-fold point, where $k = |C_{0}\cap \overline{C\backslash C_{0}}|$. 
In particular, $C'$ is a multi-nodal rational curve. $C'$ admits a degree $d$ finite map $g$ to $\PPr$ such that $\mu = g \circ f$. 
\end{lemma}

\begin{proof} 
Let $C_{0}$ denote a maximal connected subcurve of $C$ contracted by $\mu$. Suppose $C_{1},\ldots, C_{k}$ are $k$ irreducible components of $C$ attached to $C_{0}$ at $k$ distinct points. The map $\mu$ has positive degree $d_{i}$ restricted to 
$C_{i}$ for $0 < i < k$. Therefore, $L$ restricted to $C_{i}$ is a degree $d_{i}$ line bundle for $i>0$ and $L$ is trivial on $C_{0}$. 
Let $D$ be the union of $C_{0},C_{1},\ldots, C_{k}$ and let $d_{0} = \sum d_{i}$. By the same argument in Lemma~\ref{3.9}, we know that $h^{0}(D, mL) = 1+ md_{0}$. The sections of $mL$ induce a map from $D$ to $\mathbb P^{md_{0}}$ such that $C_{i}$ maps to a degree 
$md_{i}$ rational normal curve spanning a subspace $\mathbb P^{md_{i}}$ for $i>0$ and $C_{0}$ maps to a point $p$ contained in all those $\mathbb P^{md_{i}}$. Note that $md_{1} + \cdots + md_{k} = md_{0}$.
So those linear subspaces $\mathbb P^{md_{i}}$ cannot contain a common line. It implies that the tangent directions of the $k$ rational normal curves at $p$ are linearly independent. Therefore, $f$ restricted to $D$ contracts $C_{0}$ to a rational $k$-fold point. 

Since any component of $C$ contracted by $\mu$ is also contracted by $f$, $\mu$ factors through $f$ and $g$ is a finite map of degree $d$. 
\end{proof}

The picture below illustrates an example of the above situation. $\mu$ contracts $C_{0}$ and maps $C$ to three concurrent lines in a plane. But the map factors through three concurrent lines that span $\mathbb P^{3}$. 

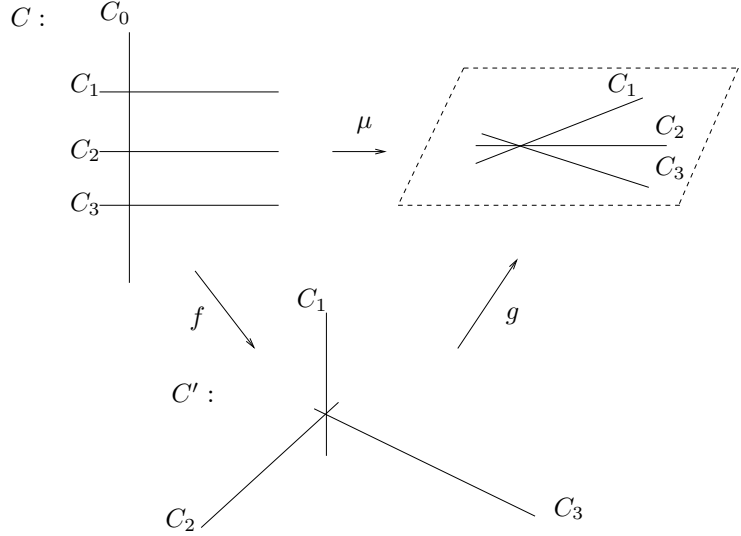
\begin{figure}[H]
\label{fold}
    \centering
    \psfrag{C0}{$C_{0}$}
    \psfrag{C1}{$C_{1}$}
    \psfrag{C2}{$C_{2}$}
    \psfrag{C3}{$C_{3}$}
    \psfrag{f}{$f$}
    \psfrag{g}{$g$}
    \psfrag{mu}{$\mu$}
    \psfrag{C}{$C:$}
    \psfrag{D}{$C':$}
    \includegraphics[scale=0.5]{fold.eps}
\end{figure}

Now let us state a global version of Lemma~\ref{local}. 

\begin{proposition}
\label{KBV}
There is a natural morphism $\pi: \Mrd \rightarrow \BV$ contracting the locus of non-finite maps with moduli. In particular, let $[C', g]$ be a branchcurve such that the domain $C'$ has $m$ rational singularities of type $k_{1}$-fold, \ldots, $k_{m}$-fold. Then $\pi^{-1}([C',g])$ is isomorphic to $\overline M_{0,k_{1}}\times \cdots \times \overline M_{0,k_{m}}$. 
\end{proposition}

\begin{proof}
Suppose there is a family of degree $d$ stable maps from rational curves to $\PPr$ over a base $B$ as follows: 
\begin{diagram}
\mathscr C              & \rTo^{\phi}  &\PPr_{B} \\
\dTo^{h}     &           & \\
B          &          & 
\end{diagram}

Let $\mathcal L = \phi^{*}\mathcal O(1)$ be the pull-back of the twisted sheaf from $\PPr_{B}$, cf. \cite[II 5]{Ha}.  
Define $f: \mathscr C\rightarrow \mathscr C' = \mbox{Proj} \big(\bigoplus_{m\geq 0} H^{0}(\mathscr C, m\mathcal L)\big)$ relative to the base $B$. Every curve contracted by $f$ is contained in a fiber of $h$ and is also contracted by $\phi$. So 
$\mathscr C'$ admits a morphism $g$ to $\PPr_{B}$ such that $\phi = g \circ f$. $\mathscr C'$ can be viewed as a family of finite maps over $B$. Below we will show that the arithmetic genus of a fiber of $\mathscr C'$ over $B$ is always zero. Then $(\mathscr C', g)$ is 
a family of branchcurves with Hilbert polynomial $1+md$. 

Let $C_{b}$ be a fiber of $\mathscr C$ over $b\in B$. $h^{0}(m\mathcal L|_{C_{b}}) = 1+md$ does not depend on $b$. By cohomology and base change, 
cf. \cite[Corollary 12.9]{Ha}, 
we get $H^{0}(\mathscr C, m\mathcal L)|_{C_{b}} \cong H^{0}(C_{b}, m\mathcal L|_{C_{b}})$. Therefore, the corresponding fiber $C'_{b}$ in $\mathscr C'$  is isomorphic to Proj \big($\bigoplus_{m\geq 0} H^{0}(C_{b}, m\mathcal L|_{C_{b}})\big)$. 
By Lemma~\ref{local}, $C_{b}$ is a multi-nodal rational curve and $g|_{C'_{b}}$ is a finite map of degree $d$. Hence, $g: \mathscr C'\rightarrow \PPr_{B}$ is a family of degree $d$ finite maps from multi-nodal rational curves to $\PPr$. The construction of this family is compatible with base change. In particular, it yields the desired morphism from $\Mrd$ to $\BV$. 

Let $[C', g]$ be a branchcurve such that the domain $C'$ has $m$ rational singularities $p_{1}, \ldots, p_{m}$ of type 
$k_{1}$-fold, \ldots, $k_{m}$-fold, respectively. Let $[C, f]$ be a stable map contained in $\pi^{-1}([C',g])$. Then over $p_{i}$, there is a rational subcurve glued to $k_{i}$ components of $C$ and contracted by $f$. We can move the $k_{i}$ components along this subcurve without affecting their image $[C',g]$ in $\BV$. Therefore, we get  $\pi^{-1}([C',g]) \cong \overline M_{0,k_{1}}\times \cdots \times \overline M_{0,k_{m}} $. It has positive dimension if and only if some of the $k_{i}$
is bigger than three. So the locus of non-finite maps with moduli gets contracted by $\pi$.
\end{proof}

\begin{proposition}
\label{NBV}
For a divisor $D = a H + bT,\  a, b > 0$ on $\Mrd$, let $\varphi: \Mrd \rightarrow M(D)$ be the corresponding contraction. Then the morphism $\pi: \Mrd \rightarrow \BV$ factors through $\varphi$. In particular, $M(D)$ is normal and bijective to $\BV$. 
\end{proposition}

\begin{proof}
By Lemma~\ref{non-finite}, a curve is contracted by $\varphi$ if and only if it is contained in the locus of non-finite maps with moduli. It must also be contracted under $\pi$ by the construction of $\pi$. Moreover, $\varphi_{*}\mathcal O_{\Mrd} \cong \mathcal O_{M(D)}$ so $\pi$ factors through $\varphi$. $\Mrd$ is $\QQ$-factorial, which implies that $M(D)$ is normal. The map $M(D)\rightarrow \BV$ is bijective since $\pi$ has connected fibers by Proposition~\ref{KBV}. Therefore, $M(D)$ is the normalization of $\BV$. 
\end{proof}

We thus completed the proof of Theorem~\ref{BV}.

\begin{remark} We do not know whether $\BV$ is normal. The two spaces $\Mrd$ and $\BV$ are different if and only if $d \geq 4$. If $d = 2$, there is no stable map that fails to be finite. If $d=3$, the only non-finite stable maps are those whose domain curves consist of a rational subcurve $C_{0}$ attached with three rational tails, where $C_{0}$ gets contracted under the map. Since three points on $\mathbb P^{1}$ have unique moduli, $\overline{\mathcal M}_{0,0}(\PPr, 3)$ is isomorphic to $\overline{\mathcal B}_{0,0}(\PPr, 3)$.  
\end{remark}

If $\pi: \Mrd \rightarrow \BV$ is a small contraction, $\BV$ is singular. For instance, it may not be $\QQ$-factorial. 
It is unclear which Weil divisors on $\BV$ are $\QQ$-Cartier. Nevertheless, it is still possible to get line bundles on $\BV$ by pushing down divisors from $\Mrd$. We will use the same notation to denote a divisor on $\Mrd$ and its image in $\BV$. 

\begin{proposition}
\label{ample}
$H$ and $T$ are $\QQ$-Cartier divisors on $\BV$. In particular, a positive linear combination of $H$ and $T$ is ample on $\BV$. 
\end{proposition}

\begin{proof}
For a family $B$ of degree $d$ finite maps from multi-nodal rational curves to $\PPr$, the vector bundle $E_{m}$ in Theorem~\ref{divisor classes} can be 
defined over $B$. It remains locally free due to the fact $h^{1}(C, f^{*}\mathcal O_{\PPr}(m)) = 0$ for a branchcurve $[C, f]$. 
Therefore, $L_{m} = $ Det $(E_{m})$ is a natural line bundle on the moduli space $\BV$. The divisor class $c_{1}(L_{m})$ on $\Mrd$ is 
$D_{m} =  (\frac{m^{2}}{12}+m) H - mT$. Moreover, $\Mrd$ and $\BV$ are isomorphic in codimension one. Hence, by varying $m$ we get both $H$ and $T$ 
as $\QQ$-Cartier divisors on $\BV$. Take $a, b > 0$ such that $D = a H + bT$ is Cartier on $\BV$. By Proposition~\ref{NBV}, 
the contraction $\Mrd \rightarrow \BV$ factors through $M(D)$. Moreover, $D$ is ample on $M(D)$ and $M(D) \rightarrow \BV$ is a bijection. 
It follows that $D = aH + bT$ as a divisor on $\BV$ is ample, cf. \cite[1.2.28, 1.2.29]{L}. 
\end{proof}

Proposition~\ref{ample} implies Corollary~\ref{PSN}. Finally, let us consider Theorem~\ref{Simp}. 

\begin{proposition}
If a divisor $D$ on $\Mrd$ is a positive linear combination of $H$ and $NL$, then $\BD$ consists of the locus of multi-image maps.
\end{proposition}

\begin{proof}
Take a multi-image map and vary a branch point on its image. All maps parameterized by such a one-dimensional family $R$ have the same image, so 
$R\ldotp H = 0$. On the other hand, if the moving branch point meets the defining hyperplane of $T$, then $R\ldotp T > 0$. This implies that 
$R\ldotp NL < 0$ since $H$ is a positive linear combination of $T$ and $NL$, cf. Theorem~\ref{divisor classes}. Hence, $R\ldotp D < 0$ and $\BD$ contains the locus of multi-image maps. 

Conversely, if a map is not a multi-image map, after a general projection to $\mathbb P^{2}$, it is still not a multi-image map. We can pick a defining line $L$ for $NL$ such that no singular point of the image is contained in $L$. Therefore, this map is not contained in the base locus of 
$NL$. We already know that $H$ is semi-ample. So this map is not contained in $\BD$. 
\end{proof}

It would be interesting to find the flipping space with respect to $D$. Unlike the case $d = r = 3$, in general the Hilbert scheme is not
the flipping space. For instance, when $d = r = 4$, a rational one-nodal quartic may have different embedded scheme structures at the node. However, the corresponding configuration in the Kontsevich space or the Chow variety is unique and does not depend on the embedded point. 
A more plausible candidate is provided by the Simpson moduli space of semi-stable sheaves with 1-dimensional support. 

\begin{definition} 
A coherent sheaf $E$ is \emph{pure} if any non-zero subsheaf of $E$ has the same dimensional support as $E$. 
A pure sheaf $E$ is \emph{stable (semi-stable)} if $\frac{\chi(E(m))}{r(E)} < (\leq) \frac{\chi(F(m))}{r(F)}, \ m \gg 0$ for any 
non-trivial pure quotient sheaf $F$ of the same dimension, where $\chi(E(m))$ is the Hilbert polynomial of $E$ and $r$ is its leading coefficient.
\end{definition}

In \cite{cSi}, it was proved that there exists a moduli space Simp$^{P}(\PPr)$ parameterizing semi-stable sheaves on $\PPr$ with Hilbert polynomial $P$. Here we are interested in the case when the sheaves have 1-dimensional support and $P$ equals $dm+1$. Let us first justify the stability of certain sheaves. 

\begin{lemma}
\label{CM}
The structure sheaf of a Cohen-Macaulay curve of arithmetic genus zero in $\PPr$ is stable. 
\end{lemma}

\begin{proof} Let $C$ be a Cohen-Macaulay curve of arithmetic genus zero in $\PPr$. 
Let $F$ be a pure quotient sheaf of $\mathcal O_{C}$. $F$ can also be regarded as the structure sheaf of a closed subcurve $C'$ of $C$. 
Suppose $C$ and $C'$ have degree $d$ and $d'$, respectively, $d > d'$. By the exact sequence 
$$ 0\rightarrow I_{C'/C}\rightarrow \mathcal O_{C}\rightarrow \mathcal O_{C'}\rightarrow 0, $$
we have $0 = h^{1}(\mathcal O_{C}) \geq h^{1}(\mathcal O_{C'}) = p_{a}(C')$. 
Therefore, 
$$\frac{\chi(\mathcal O_{C}(m))}{d} = m + \frac{1}{d} < m + \frac{1-p_{a}(C')}{d'} = \frac{\chi(\mathcal O_{C'}(m))}{d'}$$
for $m\gg 0$, which implies that $\mathcal O_{C}$ is stable. 
\end{proof}

For a smooth connected degree $d$ rational curve $C$ in $\PPr$, $r\geq 3$, the correspondence $[C]\rightarrow [\mathcal O_{C}]$ in Lemma~\ref{CM} induces an injection from the locus of maps whose images are smooth connected degree $d$ rational curves to Simp$^{dm+1}(\PPr)$. Let $S^{dm+1}(\PPr)$ denote the closure of its image in Simp$^{dm+1}(\PPr)$.  

\begin{proposition}
\label{direct image}
For $r\geq 3$, let $[C,\mu]\in \Mrd$ be a stable map which is not a multi-image map. Then $\mu_{*}\mathcal O_{C}$ is a stable sheaf represented by
a point of $S^{dm+1}(\mathbb P^{r})$. This yields a birational morphism from the complement of the locus of multi-image maps in $\Mrd$
to $S^{dm+1}(\PP^{r})$. In particular, this morphism factors through the complement of the locus of multi-image maps in $\BV$ and the resulting morphism 
to $S^{dm+1}(\PP^{r})$ is injective. 
\end{proposition}

\begin{proof}
If $\mu$ contracts a component of $C$, we can use the factorization $g: C'\rightarrow \PPr$ of $\mu$ obtained in Lemma~\ref{local} instead, 
due to the fact $\mu_{*}\mathcal O_{C} \cong g_{*}\mathcal O_{C'}.$ So from now on, assume that $\mu$ is a finite map which is generically one-to-one from 
a multi-nodal rational curve $C$ to its image $D$ in $\PPr$. 

Since $\mu$ restricted to any component of $C$ is not a multiple covering map, $D$ is of degree $d$ and 
$\mu_{*}\mathcal O_{C}$ is a purely 1-dimensional sheaf supported on $D$. Since $\mu$ is finite, 
$R^{i}\mu_{*}E = 0$ for all $i>0$ and $h^{j}(D, \mu_{*}E) = h^{j}(C, E)$ for all $j\geq 0$, where $E$ is an arbitrary coherent sheaf on $C$. 
In particular, $h^{0}(\mu_{*}\mathcal O_{C}(m)) = dm+1$ and $h^{1}(\mu_{*}\mathcal O_{C}(m)) = 0$ for all $m\geq 0$. Hence, the Hilbert polynomial 
of $\mu_{*}\mathcal O_{C}$ equals $dm+1$. 

Suppose $F$ is a non-trivial pure quotient sheaf of $\mu_{*}\mathcal O_{C}$ supported on a subcurve $D'$ of $D$ with degree $d'< d$. 
We have two exact sequences as follows: 
$$ 0\rightarrow I\rightarrow \mu_{*}\mathcal O_{C}(m)\rightarrow F(m)\rightarrow 0 ,$$ 
$$ 0\rightarrow \mathcal O_{D'}(m) \rightarrow F(m)\rightarrow Z\rightarrow 0, $$
where $Z$ is a sheaf with support in $D'$. From the first sequence we have $h^{1}(F(m)) = 0$. Combining with the second 
we get $h^{0}(Z) \geq h^{1}(\mathcal O_{D'}(m)).$ Therefore, 
$h^{0}(F(m)) = 1+md' - h^{1}(\mathcal O_{D'}) + h^{0}(Z) \geq 1+md'.$ 
Then we have 
$$\frac{\chi(F(m))}{d'} \geq m + \frac{1}{d'} > m + \frac{1}{d} = \frac{\chi(\mu_{*}\mathcal O_{C}(m))}{d},$$ 
which implies the stability of $\mu_{*}\mathcal O_{C}(m)$. 

The fact that $R^{i}\mu_{*}\mathcal O_{C}= 0$ for all $i>0$ guarantees that the above correspondence can be done for a family of 
finite generically one-to-one maps. Therefore, we get an injection from the complement of the locus of multi-image maps in $\BV$ 
to $S^{dm+1}(\PP^{r})$. Using the morphism $\Mrd\rightarrow \BV$, we also get a morphism from the complement of the locus of multi-image maps in $\Mrd$
to $S^{dm+1}(\PP^{r})$ that contracts the locus of non-finite maps with moduli. 
\end{proof}

\begin{remark}
For a multi-image map $[C,\mu]\in \Mrd$, $\mu_{*}\mathcal O_{C}$ may fail to be semi-stable. For instance, let $\mu: C\rightarrow L$ be a degree $d>1$ 
branched cover from a smooth rational curve $C$ to a line $L$ in $\PPr$. Since $\mu$ is finite, $\mu_{*}\mathcal O_{C}$ is locally free of rank $d$ on $L$
and $h^{0}(\mu_{*}\mathcal O_{C}(m)) = dm+1$ for all $m\geq 0$. $\mu_{*}\mathcal O_{C}$ splits into $\mathcal O_{L}\oplus E$ by the trace map, where $E$
is a rank $d-1$ subbundle. Then the subbundle $\mathcal O_{L}$ destabilizes $\mu_{*}\mathcal O_{C}$. 
\end{remark}

The case $d=3$ is particularly interesting. 

\begin{proposition}
\label{d=3}
For $r\geq 3$, $S^{3m+1}(\PPr)$ is the flip of $\Mrt$ over the Chow variety with respect to the divisor $D = aH + bNL, \ a, b>0$. 
In particular, $S^{3m+1}(\PPr)$ is smooth and has Picard number two. 
\end{proposition}
 
\begin{proof}
The prototype is $r=3$. In that case, it was shown in \cite[Theorem 1.1 (3)]{FT} that $S^{3m+1}(\mathbb P^{3})$
is one of the two components of Simp$^{3m+1}(\mathbb P^{3})$. Moreover, $S^{3m+1}(\mathbb P^{3})$ is smooth and isomorphic to
the closure $H^{3m+1}(\mathbb P^{3})$ of the locus of twisted cubics in the Hilbert scheme Hilb$^{3m+1}(\mathbb P^{3})$. In \cite{C}, it was shown that the Hilbert component $H^{3m+1}(\mathbb P^{3})$ is the flip of $\Mthth$ over the Chow variety with respect to the divisor class $aH + bNL, \ a, b > 0$. 

For $r\geq 4$, let $H^{3m+1}(\PPr)$ be the closure of the locus of 
smooth rational cubic curves in the Hilbert scheme Hilb$^{3m+1}(\mathbb P^{r})$. 
In \cite[Proposition 1.3]{CK}, it was proved that the morphism $H^{3m+1}(\PPr)\rightarrow S^{3m+1}(\PPr)$ is 
the blow-up along the locus of stable sheaves with planar support. Note that $H^{3m+1}(\PPr)$ is a fiber bundle over 
the Grassmannian $\mathbb G(3,r)$ with fiber isomorphic to $H^{3m+1}(\mathbb P^{3})$. Since $H^{3m+1}(\mathbb P^{3})$ has Picard 
number two, the Picard numbers of $H^{3m+1}(\PPr)$ and $S^{3m+1}(\PPr)$ are three and two, respectively. 
By Proposition~\ref{direct image}, the birational map $\Mrt \dashrightarrow S^{3m+1}(\PPr)$ is an isomorphism along the complement of the locus of multi-image maps. We still denote by $D$ the birational transform of a divisor $D$ from 
$\Mrt$ to $S^{3m+1}(\PPr)$ or $H^{3m+1}(\PPr)$. The divisor class $H$ is semi-ample on all 
three spaces. Moreover, for a sheaf $[F]$ in $S^{3m+1}(\PPr)$, it can be regarded as a sheaf in a subspace $\mathbb P^{3}$. 
The projection from $\PPr$ to this $\mathbb P^{3}$ induces a rational map $\pi: S^{3m+1}(\PPr)\dashrightarrow S^{3m+1}(\mathbb P^{3})$. $\pi$ is 
a morphism in codimension one and well-defined in a local neighborhood of $[F]$. Since $NL$ is base-point-free on $S^{3m+1}(\mathbb P^{3})$, 
its pull-back under $\pi$ does not contain $[F]$ in the base locus. Hence, as a divisor on $S^{3m+1}(\PPr)$, $NL$ is also base-point-free.   
Since $S^{3m+1}(\PPr)$ has Picard number two and $HL, H$ are both semi-ample, its ample cone is bounded by $NL$ and $H$.   
Therefore, $S^{3m+1}(\PPr)$ is the flip of $\Mrt$ with respect to a divisor $aH + bNL, \ a, b>0$. 
\end{proof}

Proposition~\ref{d=3} completes the proof of Theorem~\ref{Simp} for $r\geq 3$. For $r=2$, it was shown in \cite{LP} that Simp$^{3m+1}(\mathbb P^{2})$ is isomorphic to the universal cubic over $\mathbb P^{2}$. A rational plane cubic $C$ has a node $p$. We associate $C$ with a coherent sheaf $E$ 
such that there exists a non-split extension 
$$0\rightarrow \mathcal O_{C}\rightarrow E\rightarrow \mathbb C_{p}\rightarrow 0, $$
cf. \cite[Lemma 3.2]{FT}. Note that $E$ is isomorphic to the direct image sheaf $\mu_{*}(\mathcal O_{\widetilde C})$, where
$\mu: \widetilde{C}\rightarrow C$ is the normalization of $C$. Via this, $S^{3m+1}(\mathbb P^{2})$ as a subvariety of Simp$^{3m+1}(\mathbb P^{2})$ is isomorphic to the incidence correspondence $\{(p, C)|p\in C_{sing}\}$, where $C$ is a singular plane cubic. We discussed at the beginning of this section 
that the incidence correspondence $S^{3m+1}(\mathbb P^{2})$ is the model $M(D)$ for a divisor $D = aH + bNL, \ a, b > 0$. $M(D)$ is the flipping space of $\Mtth$ over the Chow variety with respect to $D$. This proved Theorem~\ref{Simp} for $r=2$. Overall, what we proved can be summarized as a wall-crossing phenomenon that flips the ample cone $(H,T)$ of $\Mrt$ to the ample cone $(H,NL)$ of $S^{3m+1}(\PP^{r})$. 

A variety on which Mori's program can be carried out for every effective divisor is called a Mori dream space, cf. \cite{HK}. Equivalently, a Mori dream space is a space whose Cox ring is finitely generated. By \cite[Corollary 1.3.1]{BCHM}, a log Fano variety is a Mori dream space. For $\Mdd$, $-K_{\Mdd}$ is big since it is proportional to $D_{deg} + \frac{2}{d+1}\Delta$, cf. Theorem~\ref{divisor classes}. But it is not ample except for very small 
$r$ when $d$ is two or three. In general, we can ask the following fundamental question. 
\begin{question}
Is $\Mdd$ a Mori dream space? 
\end{question}

\begin{remark}
The assumption that $-K_{X}$ is big does not guarantee that a variety $X$ is a Mori dream space. For instance, let $X$ be the blow-up of $\mathbb P^{3}$ at 12 general points on an plane cubic $C$. Denote $H$ as the pullback of $\mathcal O(1)$ and $E$ as the sum of the exceptional divisors. In this case $-K_{X} = 4H-2E = 2H + 2(H-E)$ is big. Moreover, the divisor $D = 4H - E$ is big and nef, but it has the proper transform of $C$ contained in its stable base locus. The section ring $R(D)$ is not finitely generated, since for a big and nef divisor, its section ring is finitely generated if and only if it is semi-ample. 
\end{remark}

\section{Volume of a divisor}
Let $D$ be an effective $\QQ$-Cartier divisor on a variety $X$ of dimension $n$. The volume of $D$ is an important numerical invariant defined by the expression 
$$ \mbox{vol}(D) =\lim_{m \to\infty} \mbox{sup} \frac{h^{0}(X, mD)}{m^{n}/n!}. $$
If $D$ is nef, vol($D$) $= D^{n}$, cf. \cite[2.2.C]{L}. The study of the effective cone decomposition and birational models of $\Mrd$ yields a useful tool for calculating vol($D$). 

To illustrate the idea, consider the Kontsevich space $\Mtth$ discussed in section 2. $\Mtth$ is of dimension 8. Its effective cone is bounded by $\Delta$ and $NL$. $H$ and $T$ further decompose the cone into 
three Mori Chambers. 
If a divisor $D$ lies in the nef cone $\langle H, T \rangle$, vol($D$) = $D^{8}$ can be calculated using Pandharipande's algorithm, cf. 
\cite[4.2]{P1}. If $D$ is in the chamber $\langle T, \Delta \rangle$, then $D = a T + b\Delta$, $a, b > 0$ and 
$\mbox{vol}(D) = a^{8}\mbox{vol}(T)$. If $D$ is inside the chamber $\langle NL,  H\rangle$, $D$ is not nef and $\BD$ contains the locus of multi-image maps. However, we can flip $\Mtth$ to the birational model $M(D) = \{ (p, C) \}\subset \mathbb P^{2}\times \mathbb P^{9}$, where $p$ belongs to the singular locus of a plane nodal cubic or its degeneration. Denote by $\widetilde{A}$ the birational transform of a divisor $A$ from $\Mtth$ to $M(D)$. On $M(D)$, $\widetilde{D}$ is ample. Since $\Mtth$ and $M(D)$ are isomorphic in codimension 1, vol($D$) = vol($\widetilde{D}$) = $\widetilde{D}^{8}$. Write $\widetilde{D}$ as a linear combination of 
$\widetilde{H}$ and $\widetilde{NL}$ on $M(D)$. In order to compute $\widetilde{D}^{8}$, it suffices to know all the top intersections $\widetilde{H}^{k}\widetilde{NL}^{8-k}$. $M(D)$ has the structure of a $\mathbb P^{6}$-bundle over $\mathbb P^{2}$ as follows. Let $E_{k}$ be a vector bundle over $\mathbb P^{2}$ whose fiber over a point 
$p$ corresponds to the space of plane cubics with vanishing order $\geq k$ at $p$. Then $M(D)$ is isomorphic to $\mathbb P E_{2}$. Let 
$\pi_{1}$ and $\pi_{2}$ be the projections from $\mathbb PE_{2}$ to $\mathbb P^{2}$ and to $\mathbb P^{9}$ the space of plane 
cubics, respectively. Let $S$ be the tautological line bundle with Chern class $1-\eta$ on $\mathbb PE_{2}$ and let $Q$ be the quotient bundle. 
Let $l$ denote the pullback of a line class via $\pi_{1}$. We have the following exact sequences: 
$$ 0\rightarrow S \rightarrow \pi_{1}^{*} E_{2} \rightarrow Q \rightarrow 0 , $$
$$ 0\rightarrow E_{2} \rightarrow E_{1} \rightarrow T^{*}_{\mathbb P^{2}}\otimes \mathcal O_{\mathbb P^{2}}(3)\rightarrow 0, $$
$$ 0\rightarrow E_{1} \rightarrow E_{0} \rightarrow \mathcal O_{\mathbb P^{2}}(3) \rightarrow 0.$$ 
A Chern class calculation shows that 
$$ c(\pi_{1}^{*} E_{2}) = 1 - 6 l + 24 l^{2} $$ and 
$$ c(Q) = \frac{1-6 l + 24 l^{2}}{1-\eta}. $$
Since $Q$ has rank six, the relation $c_{7}(Q) = 0$ yields 
$$ \eta^{7} - 6\eta^{6}l + 24 \eta^{5}l^{2} = 0. $$ 
Since $\eta^{6}l^{2} = 1$ and $l^{3} = 0$, we get 
$\eta^{7}l = 6 $ and $\eta^{8} = 12$. 
The class $\widetilde{NL}$ is just $l$. Moreover, $S$ can be identified as 
$\pi_{2}^{*}\mathcal O_{\mathbb P^{9}}(-1)$, which 
implies $\widetilde{H} = \eta$. Now the top intersection numbers of $\widetilde{H}$ and $\widetilde{NL}$ are: 
$$ \widetilde{H}^{8} = 12, \ \widetilde{H}^{7}\widetilde{NL} = 6, \ \widetilde{H}^{6}\widetilde{NL}^{2} = 1, \ \widetilde{H}^{a}\widetilde{NL}^{8-a} = 0 \ \mbox{for}\ 0\leq a\leq 5. $$

The situation for $\Mthth$ is similar. $\Mthth$ is of dimension 12. 
Its effective cone is bounded by $D_{deg}$ and $\Delta$. $H, T$ and $NL$ decompose the cone into Mori chambers. 
If a divisor $D$ lies in the nef cone $\langle H, T \rangle$, vol($D$) = $D^{12}$ can be calculated using Pandharipande's algorithm. 
If $D$ is in the chamber $\langle T, \Delta \rangle$, $\BD$ contains $\Delta$. Then $D = aT + b\Delta$, $a, b>0$ and 
$\mbox{vol}(D) = a^{12} \mbox{vol}(T)$. If $D$ is in the chamber $\langle D_{deg}, NL\rangle $, $\BD$ contains $D_{deg}$. Then $D= aNL + bD_{deg}$, $a, b>0$ and $\mbox{vol}(D) = a^{12} \mbox{vol}(NL)$. The remaining chamber is $\langle NL, H\rangle$. A divisor $D$ in this chamber is not nef and $\BD$ contains the locus of multi-image maps.
But we have seen that the corresponding Hilbert component $H^{3m+1}(\mathbb P^{3})$ is the flip of $\Mthth$ over the Chow variety. Denote by $\widetilde A$
as the birational transform of a divisor $A$ from $\Mthth$ to $H^{3m+1}(\mathbb P^{3})$. The chamber $\langle \widetilde{NL}, \widetilde{H}\rangle$ is the nef cone of $H^{3m+1}(\mathbb P^{3})$. 
So vol($D$) = vol($\widetilde{D}$) = $\widetilde{D}^{12}$ for a divisor in this chamber. Write $D$ as a linear combination of $H$ and $NL$. 
Since $H$ is nef on both $\Mthth$ and $H^{3m+1}(\mathbb P^{3})$, $\widetilde{H}^{12} = H^{12} = 80160$, cf. \cite[4.3]{P1}. Moreover, $NL$ is numerically equivalent to $H + D_{deg}$. So the top intersection 
of $\widetilde{H}$ and $\widetilde{NL}$ can be worked out once we know the top intersection of $\widetilde{H}$ and $\widetilde{NL}$ restricted to $\widetilde{D_{deg}}$. $\widetilde{D_{deg}}$ has a simple structure. It is isomorphic to a $\mathbb P^{6}$-bundle over the point-plane flag variety $F = \{ p\in \mathbb P^{2}\subset \mathbb P^{3}\}$. The fiber over a point $(p,\mathbb P^{2})$ in $F$
parameterizes the space of nodal cubics on that plane and their degenerations that contain $p$ as a singular point. There is a unique embedded point supported at $p$ on those cubics. The flag variety $F$ admits natural projections to $\mathbb P^{3}$ and $\mathbb P^{3*}$. Use $\lambda$ and $\kappa$ to denote the pull-back of $\mathcal O(1)$ from $\mathbb P^{3}$ and $\mathbb P^{3*}$ to $\widetilde{D_{deg}}$, respectively. Denote by $-\eta$ the first Chern class of the tautological line bundle of $\widetilde{D_{deg}}$. Note that $\lambda, \kappa$ and $\eta$ generate the Chow ring of $\widetilde{D_{deg}}$. By using similar exact sequences as in the last paragraph, these classes satisfy the following relations: 
$$\lambda^4=\kappa^4=0, \ \kappa^3-\kappa^2\lambda+\kappa\lambda^2-\lambda^3=0, $$
$$ \eta^7+(9\kappa-6\lambda)\eta^6+(45\kappa^2-52\kappa\lambda+24\lambda^2)\eta^5 $$
$$+(85\lambda^3+35\lambda^2\kappa-85\lambda\kappa^2)\eta^4+(40\kappa^2\lambda^2+240\kappa\lambda^3)\eta^3+280\kappa^2\lambda^3\eta^2 = 0. $$
Their top intersection numbers on $\widetilde{D_{deg}}$ can be worked out from the above relations. Moreover, using test curves in $\widetilde{D_{deg}}$, 
one can check that $\widetilde{H}$ and $\widetilde{NL}$ restricted to $\widetilde{D_{deg}}$ are numerically equivalent to $\eta+3\kappa$ and $\lambda + 3\kappa$, respectively. Hence, we obtain the following intersection
numbers: 
$$ \widetilde{H}^{12} = 80160, \ \widetilde{H}^{11}\widetilde{NL} = 93120, \ \widetilde{H}^{10}\widetilde{NL}^{2} = 104280, $$ $$\widetilde{H}^{9}\widetilde{NL}^{3} = 112360, \ \widetilde{H}^{8}\widetilde{NL}^{4} = 116896, \ \widetilde{H}^{7}\widetilde{NL}^{5} = 118660,$$
$$\widetilde{H}^{a}\widetilde{NL}^{12-a} = 119020 \ \mbox{for}\ 0\leq a\leq 6. $$

\section{Appendix (Charley Crissman)} 
A Macaulay 2 program is included to verify that $C_{3}$ and $C_{2}$ used in the proof of Theorem~\ref{EB} are moving curves 
on $\Mthf$ and $\Mtf$, respectively. 

Reinterpret the construction of $C_{3}$ as follows. Consider the space of polynomials of bidegree up to 7 and 4 in two
variables $x$ and $y$. It is a 40-dimensional vector space. Take 12 random pairs
$p_{i} = (x_{i}, y_{i})$. Consider the subspace $U$ of polynomials that have a zero of order at least two at each $p_{i}$. Since a double point imposes 3 conditions, the expected dimension of $U$ is $40-12\times 3 = 4$ for a general choice of the points $p_{i}$. Restrict the
polynomials in $U$ to $x = 0$. Then we get a 4-dimensional space $V$ of
degree 4 polynomials in the single variable $y$. Since the space of polynomials of degree 4 in $y$ is 5-dimensional, $V$ can be regarded as a point $[V]$ in the Grassmannian $G(4,5)$. We thus get a map from the union of the points $p_{i}$ to $G(4,5)$. 
To check that $C_{3}$ is moving on $\Mthf$, it suffices to verify that if we vary the points $p_{i}$, the
corresponding $[V]$ covers an open dense subset of $G(4,5)$, namely, the associated map of differentials is onto. 
The Macaulay 2 code to verify this is the following: 
\begin{verbatim}
kk = ZZ/32003

R = kk[u,x,y,z]
 
P = random(R^2,R^12) 

M = matrix({flatten(entries(matrix(apply(5, j->{y^j*z^(4-j)}))
*matrix({apply(8, j-> x^j*z^(7-j))})))})

for i from 0 to 11 do A_i = (ideal(x - z*P_(0,i), y - z*P_(1,i)))^2 

for i from 0 to 11 do (
D = intersect(if i==0 then ideal(1_R) else intersect(apply(i, j-> A_j)), 
if i==11 then ideal(1_R) else intersect(apply(11-i, j->A_(i+j+1))), 
(ideal(x-z*u - z*P_(0,i),y-z*P_(1,i)))^2);
S = R^1/module(D); T = R^40; F = map(S, T, M); K = mingens(kernel(F)); 
V_i = apply(5, j->apply(4, k->(K)_(8*j, 95-k)))
)

Y = apply(12, i-> flatten(entries(exteriorPower(4,matrix(V_i))))) 

5 == rank(sub(diff(u, matrix(Y)), u=>0))
\end{verbatim}

For $C_{2}$, consider the space of polynomials of bidegree up to 5 and 4 in two
variables $x$ and $y$. It is a 30-dimensional vector space. Take 9 random pairs
$p_{i} = (x_{i}, y_{i})$. Consider the subspace $U$ of polynomials that have a zero of order at least two at each $p_{i}$. The expected dimension of $U$ is $30-9\times 3 = 3$ for a general choice of the points $p_{i}$. Restrict the
polynomials in $U$ to $x = 0$. Then we get a 3-dimensional space $V$ of
degree 4 polynomials in the single variable $y$. Since the space of polynomials of degree 4 in $y$ is 5-dimensional, $V$ is represented by a point $[V]$ in the Grassmannian $G(3,5)$. We thus get a map from the union of the points $p_{i}$ to $G(3,5)$. 
To check that $C_{2}$ is moving on $\Mtf$, it suffices to verify that the associated map of differentials is onto. 
The Macaulay 2 code to verify this is the following:
\begin{verbatim}
kk = ZZ/32003

R = kk[u,x,y,z]
 
P = random(R^2,R^9) 

M = matrix({flatten(entries(matrix(apply(5, j->{y^j*z^(4-j)}))
*matrix({apply(6, j-> x^j*z^(5-j))})))})

for i from 0 to 8 do A_i = (ideal(x - z*P_(0,i), y - z*P_(1,i)))^2

for i from 0 to 8 do (
D = intersect(if i==0 then ideal(1_R) else intersect(apply(i, j-> A_j)), 
if i==8 then ideal(1_R) else intersect(apply(8-i, j->A_(i+j+1))), 
(ideal(x-z*u - z*P_(0,i),y-z*P_(1,i)))^2);
S = R^1/module(D); T = R^30; F = map(S, T, M); K = mingens(kernel(F)); 
V_i = apply(5, j->apply(3, k->(K)_(6*j, 69-k))) 
)

Y = apply(9, i-> flatten(entries(exteriorPower(3,matrix(V_i)))))

7 == rank(sub(diff(u, matrix(Y)), u=>0))
\end{verbatim}

For polynomials of bidegree $(a,b)$ with $k_{i}$ points of given vanishing order $m_{i}$, in case they yield a 1-dimensional family 
of maps in $\Mrd$, where $d=a$ and $r = (a+1)(b+1) - \sum {m_{i}+1\choose 2} - 1$, a general program has been written to check 
whether the family is moving on $\Mrd$, cf. \cite{Cr}. This provides a useful tool for understanding the effective cone of $\Mrd$ for larger $d$.

\end{document}